\documentclass[12pt]{article}
\usepackage{amsmath}
\usepackage{amsfonts}
\usepackage{amsthm}
\usepackage{graphicx}
\usepackage{overpic}
\usepackage{a4wide}
\usepackage{amssymb} 
\usepackage{pictex}
\usepackage{rotating}  
\usepackage{color}
\usepackage{etex} 
\usepackage{enumitem}

\definecolor{refkey}{rgb}{0,0,1}
\definecolor{labelkey}{rgb}{1,0,0}

\usepackage{comment}

\numberwithin{equation}{section}

\newtheorem{theorem}{Theorem}[section]
\newtheorem{proposition}[theorem]{Proposition}
\newtheorem{lemma}[theorem]{Lemma}
\newtheorem{corollary}[theorem]{Corollary}
\newtheorem{Definition}[theorem]{Definition}
\newenvironment{definition}{\begin{Definition}\rm}{\end{Definition}}
\newtheorem{Remark}[theorem]{Remark}
\newenvironment{remark}{\begin{Remark}\rm}{\end{Remark}}
\newtheorem{Example}[theorem]{Example}
\newenvironment{example}{\begin{Example}\rm}{\end{Example}}
\newtheorem{RHproblem}[theorem]{RH problem}

\usepackage{color}

\newcommand{\C}{\mathbb{C}}

\newcommand{\R}{\mathbb{R}}

\renewcommand{\AA}{\mathcal A}

\newcommand{\II}{\mathcal I}

\newcommand{\MM}{\mathcal M}
\newcommand{\PP}{\mathcal P}

\newcommand{\OO}{\mathcal{O}}

\renewcommand{\Re}{{\rm Re} \,}
\renewcommand{\Im}{{\rm Im} \,}

\renewcommand{\hat}{\widehat}
\renewcommand{\tilde}{\widetilde}

\usepackage[bookmarksopen, naturalnames]{hyperref}

\begin{document}
\title{Logarithmic potential theory and large deviation}
\author{T. Bloom, N. Levenberg, and F. Wielonsky}
\maketitle \noindent
{\bf Authors:}\\[\baselineskip]
T. Bloom, bloom@math.toronto.edu\\
University of Toronto, Toronto, Ontario M5S 2E4 CANADA
\\[\baselineskip]
N. Levenberg, nlevenbe@indiana.edu\\
Indiana University, Bloomington, IN 47405 USA\\
Phone: 812-855-0957 \ \ FAX: 812-855-0046
\\[\baselineskip]
F. Wielonsky, franck.wielonsky@univ-amu.fr\\
Universit\'e Aix-Marseille, CMI 39 Rue Joliot Curie\\
F-13453 Marseille Cedex 20, FRANCE 
\\[\baselineskip]
\begin{abstract}
We derive a general large deviation principle for a canonical sequence of probability measures, having its origins in random matrix theory, on unbounded sets $K$ of $\C$ with weakly admissible external fields $Q$ and very general measures $\nu$ on $K$. For this we use logarithmic potential theory in $\R^{n}$, $n\geq 2$, and a standard contraction principle in large deviation theory which we apply from the two-dimensional sphere in $\R^{3}$ to the complex plane $\C$.
\end{abstract}
\vspace{1cm}
\hskip12pt {\it This paper is dedicated to our good friend and colleague, Ed Saff.}\\[\baselineskip]
\newpage
\section{Introduction and main results} 

Let $K$ be a closed subset of the complex plane $\C$ and $\nu$ a measure on $K$. For $k=1,2,...$, we will be concerned with the following ensemble of probability measures on $K^{k+1}$:

$$\frac{1}{Z_k}|VDM(z_0,...,z_k)|^{2\beta}\exp \Big(-2k[Q(z_0)+\cdots + Q(z_k)]\Big)d\nu (z_0)...d\nu (z_k).$$

Here \begin{itemize}
\item $Z_k$ is a normalization constant;
\item $VDM(z_0,...,z_k)=\prod_{0\leq  i<j\leq k}(z_j-z_i)$
 is the usual Vandermonde determinant;
\item $Q:K\rightarrow (-\infty,+\infty]$ is a lower semicontinuous function; and
\item $\beta >0$.
  \end{itemize}

\noindent These probability measures occur in random matrix theory as the joint probability of eigenvalues and also in the theory of Coulomb gases, where $z_0,...,z_k$ are the positions of particles. They have been extensively studied but generally only when $\nu$ is Lebesgue measure (cf., \cite{AGZ} or \cite{hpbook}).

We will deal with the global behavior as $k\rightarrow \infty$. In particular, we study the almost sure convergence of the empirical measure of a random point $\frac{1}{k+1}\sum_{i=0}^k\delta_{z_i}$ to the equilibrium measure given by the unique minimizer of the weighted energy functional; i.e., 
$$\inf \{I^Q(\mu): \ \mu\in \mathcal{M}(K)\}$$
where $\mathcal{M}(K)$ are the probability measures on $K$ and 
$$I^Q(\mu)=\int_K\int_K\log \frac{1}{|z-t|w(z)w(t)}d\mu(z)d\mu(t),$$ with $w(z)=\exp(-Q(z))$. We will also establish  a large deviation principle (LDP).

Ben Arous and Guionnet \cite{BG}, building on work of Voiculescu, first proved  a large deviation principle for the Gaussian Unitary Ensemble. This was subsequently extended to general unitary invariant ensembles. Hiai and Petz \cite{hpbook} extended these methods to 
the complex plane and {\it strongly admissible} (see Definition \ref{admit}) continuous weights $Q$. In these settings, $\nu$ was taken to be Lebesgue measure. 

More recently, the case of {\it weakly admissible weights} (see Definition \ref{admit}) on unbounded subsets of the plane was studied in \cite{H} and the existence of a unique minimizer of the weighted energy functional (which in this case may not have compact support) was established. In \cite{[ESS]} a large deviation principle was established for $Q$ weakly admissible, continuous on $\R$ or $\C$ and $\nu$ the Lebesgue measure. Such weights occur in certain ensembles (see \cite{[ESS]}, the Cauchy ensemble) and in certain vector energy problems (see \cite{H}).

In this paper we will systematically develop  the case when $Q$ is lower semicontinuous, weakly admissible and $\nu$ is more general than Lebesgue measure. We will use the methods of \cite{VELD} which first of all give the almost sure convergence of the empirical measure of a random point, and, subsequently, we obtain a large deviation principle.

The paper is organized as follows: in the next section, we give some basic results on logarithmic potential theory in $\R^n$ valid for $n\geq 2$. Using the results in $\R^{3}$ together with inverse stereographic projection from a two-dimensional sphere ${\bf S}$ to the complex plane, in section 3 we readily extend some classical potential-theoretic results valid for compact subsets of $\C$ to closed, unbounded sets with weakly admissible weights.

In sections 4-7, we return to the setting of compact sets $K$ in $\R^n$ and admissible weights (see Definition \ref{realad}; such weights need only be lower semicontinuous). Corollary \ref{416} establishes the almost sure convergence of the empirical measure of a random point to the equilibrium measure in this setting, for appropriate measures $\nu$.

 Our next  goal is to show that two functionals $\underline J$ and $\overline J$ on the space $\mathcal M(K)$ of probability measures coincide. These functionals are defined as asymptotic $L^2(\nu)-$averages of Vandermonde determinants with respect to a {\it Bernstein-Markov} measure $\nu$ on $K$. As in previous work (cf., \cite{VELD} and \cite{PELD}), weighted versions of these functionals are of essential use (Theorem \ref{rel-J-E}). This equality immediately yields a large deviation principle in this $\R^n$ setting, Theorem \ref{ldp}, in which the rate function is given in terms of the weighted energy functional independent of the Bernstein-Markov measure $\nu$.

In section 8, we deal with compact subsets of the sphere in $\R^3$ and measures of infinite mass, again establishing a LDP (Theorem \ref{ldp-S}). Measures of infinite mass arise as the push-forward of measures on unbounded subsets of the plane under stereographic projection.

Our ultimate goal, achieved in sections 9 and 10, is to utilize the $\R^n$ result to prove the analogous equality of the appropriate $J-$functionals for probability measures on closed, unbounded sets in $\C$ allowing weakly admissible weights and very general measures of infinite mass (Theorem \ref{85}). Then, via a contraction principle, we obtain  an LDP (Theorem \ref{ldp2}):  

\begin{theorem} Let $K\subset \C$ be closed, and let $Q$ be a weakly admissible weight on $K$. Assume $(K,\nu,Q)$ satisfies the weighted Bernstein-Markov property (\ref{unboundedBM}).
If $\nu$ has finite mass, assume that $(K,\nu)$ satisfies a strong Bernstein-Markov property while if $\nu$ has infinite mass in a neighborhood of infinity, assume that (\ref{cond-nu}) and (\ref{cond-eps}) are satisfied for some function $\epsilon(z)$.
Define a sequence $\{\sigma_k\}$ of probability measures on $\mathcal M(K)$ by
$$\sigma_k(G) = \frac{1}{Z_k}\int_{\tilde G_k} |VDM(z_0,..., z_k)|^2\prod_{i=0}^k e^{-2kQ(z_i)}\prod_{i=0}^k d\nu(z_i)$$ 
where $\tilde G _k= \{ (z_0,..., z_k)\in K^{k+1} : \frac{1}{k+1}\sum \delta_{z_i}\in G\}$. Then $\{\sigma_k\}$ satisfies a {\bf large deviation principle} with speed $k^{2}$ and good rate function $\mathcal I:=\mathcal I_{K,Q}$ where, for $\mu \in \mathcal M(K)$,
$$
\mathcal I(\mu)=I^Q(\mu)-I^Q(\mu_{K,Q}).
$$
 \end{theorem}

In section 11 we extend this result to the case of general $\beta$ (Theorem \ref{ldp-free}).   Our results include the LDP for a number of ensembles occurring in the literature (see Remark \ref{appldp}) and also the results of Hardy \cite{H} for Lebesgue measure in $\R$ or $\C$ (see the discussion after Theorem \ref{ldp-free}). The idea of using inverse stereographic projection and working in $\R^3$ to obtain an LDP for unbounded sets in $\C$ comes from this work.

\section{Logarithmic Potential Theory in $\R^n$}
Let $K\subset \R^n$ be compact and let $\mathcal M(K)$ be the set of probability measures on $K$ endowed with the topology of weak convergence from duality with continuous functions. We consider the logarithmic energy minimization problem:
$$\inf_{\mu\in \mathcal M(K)}I(\mu)$$
where
$$I(\mu):=\int_K \int_K \log \frac{1}{|x-y|}d\mu(x) d\mu(y)$$ 
is the logarithmic energy of $\mu$. We will say that $K$ is {\it log-polar} if $I(\mu)=\infty$ for all $\mu \in \mathcal M(K)$. It is known that any compact set of positive Hausdorff dimension is non log-polar \cite{Ca}. For a Borel set $E\subset \R^n$ we will say $E$ is log-polar if every compact subset of $E$ is log-polar. We write
$$U^{\mu}(x):= \int_K \log \frac{1}{|x-y|}d\mu(y)$$
for the logarithmic potential of $\mu$. It is locally integrable and superharmonic in all of $\R^{n}$.

We gather known results about logarithmic potentials in $\R^{n}$ in the next theorem.
\begin{theorem}\label{basics}
The following results, whose precise statements can be found in \cite{ST} for logarithmic potentials in $\C=\R^{2}$, hold true for logarithmic potentials in $\R^n$, $n\geq2$:
\begin{enumerate}[noitemsep,nolistsep]
\item
for $\mu=\mu_{1}-\mu_{2}$ a signed measure with compact support and total mass zero, with $\mu_{1}$ and $\mu_{2}$ of finite energies, $I(\mu)$ is nonnegative
and is zero if and only if $\mu_{1}=\mu_{2}$.
\item
principle of descent and lower envelope theorem (with ``q.e.'' in the latter replaced by ``off of a log-polar set'');
\item
maximum principle;
\item
continuity principle.
\end{enumerate}
\end{theorem}
\begin{proof}
The version of Item 1. in $\C$ is \cite[Lemma 1.8]{ST}. In $\R^{n}$, it follows from \cite[Theorem 2.5]{CKL}. An extension of item 1. in case of unbounded support and whenever $I(\mu)$ is well-defined is given in \cite{M}, see Example 3.3.
One checks that the proofs of the principle of descent and lower envelope theorem in $\C$, Theorems I.6.8. and I.6.9. of \cite{ST}, are valid in $\R^{n}$. Items 3. and 4. are Theorems 5.2 and 5.1 of \cite{HaK}. A maximum principle restricted to the two dimensional sphere also follows as a particular case of \cite[Theorem 5]{BDS}.
\end{proof}

We will need to work in a weighted setting. We caution the reader that, unlike the setting of compact sets in $\R^n$ where we have a single notion of admissibility for a weight function, when we work on unbounded sets in $\C$ in the next section we will have several different notions.
\begin{definition}\label{realad}
Given a compact set $K\subset \R^n$ which is not log-polar, let $Q$ be a lower semicontinuous function on $K$ with $\{x\in K: Q(x) <\infty\}$ not log-polar. We call such $Q$ {\it admissible} and write $Q\in \AA(K)$. We define $w(x):=e^{-Q(x)}$.
\end{definition} 

\noindent We refer to either $Q$ or $w$ as the {\it weight}; in \cite{ST} this terminology is reserved for $w$.

We consider now the weighted logarithmic energy minimization problem:
$$\inf I^Q(\mu),\quad\mu\in \mathcal M(K),$$
where
$$I^Q(\mu):=\int_K \int_K \log \frac{1}{|x-y|w(x)w(y)}d\mu(x) d\mu(y)=I(\mu)+2\int_KQ(x)d\mu(x).$$ Following the arguments on pp. 27-33 in \cite{ST}, we have the following.

\begin{theorem}\label{frost}
For $K\subset \R^n$ compact and not log-polar, and $Q\in \AA(K)$, 
\begin{enumerate}[noitemsep,nolistsep]
\item $V_w:=\inf_{\mu\in \mathcal M(K)}I^Q(\mu)$ is finite;
\item there exists a unique weighted equilibrium measure $\mu_{K,Q} \in \mathcal M(K)$ with $I^Q(\mu_{K,Q})=V_w$ and the logarithmic energy $I(\mu_{K,Q})$ is finite;
\item the support $S_w:=$supp$(\mu_{K,Q})$ is contained in $\{x\in K: Q(x) <\infty\}$ and $S_w$ is not log-polar;
\item Let $F_w:=V_w -\int_K Q(x)d\mu_{K,Q}(x)$ denote the (finite) Robin constant. Then
\begin{align*}
U^{\mu_{K,Q}}(x)+Q(x) & \geq F_w\text{ on }K\setminus P\text{ where }P
\text{ is log-polar (possibly empty);}\\
U^{\mu_{K,Q}}(x)+Q(x) & \leq F_w\text{ for all }x\in S_w.
\end{align*}
\end{enumerate}
\end{theorem}

\begin{remark} In the proof of the Frostman-type property 4. in \cite{ST}, one simply replaces ``q.e.'' -- off of a set of positive logarithmic capacity in $\C$ -- by ``off of a log-polar set'' as the essential property used is the existence of a measure of finite logarithmic energy on a compact subset of a set of positive logarithmic capacity in $\C$. We should mention that, in the unweighted case, the existence portion of 2. and property 4. can be found in \cite{HaK}, Theorems 5.4 and 5.8.
\end{remark}

Next we discretize: for $k\geq 2$, let the $k$-th weighted diameter $\delta_{k}^{Q}(K)$ be defined by
$$\delta_k^Q(K):=\sup_{x_1,...,x_k\in K} 
|VDM_k^Q(x_1,...,x_k)|^{2/k(k-1)},$$
where $|VDM_k^Q(x_1,...,x_k)|$ denotes the weighted Vandermonde:
\begin{align}\notag
|VDM_k^Q(x_1,...,x_k)| & :=\prod_{i<j} |x_i-x_j|w(x_i)w(x_j)= \prod_{i<j} |x_i-x_j|\prod_{j=1}^kw(x_j)^{k-1}\\\label{VDM-Q}
& =:|VDM_k(x_1,...,x_k)|\cdot \prod_{j=1}^kw(x_j)^{k-1}.
\end{align}
By the uppersemicontinuity of $(x_1,...,x_k)\to \prod_{i<j} |x_i-x_j|w(x_i)w(x_j)$ on $K^k$ the supremum is attained; we call any collection of $k$ points of K at which the maximum is
attained {\it weighted Fekete points} of order $k$ for $K,Q$. Following the proofs of Propositions 3.1--3.3 of \cite[Section 3]{VELD} we may derive similar results in $\R^{n}$.

\begin{theorem} \label{sec3} Given $K\subset \R^n$ compact and not log-polar, and $Q\in \AA(K)$,
\begin{enumerate}
\item if $\{\mu_k=\frac{1}{k}\sum_{j=1}^k\delta_{x_j^{(k)}}\}\subset \mathcal M(K)$ converge weakly to $\mu\in \mathcal M(K)$, then 
\begin{equation}\label{upboundVDM}
\limsup_{k\to \infty} |VDM_k^Q(x_1^{(k)},...,x_k^{(k)})|^{2/k(k-1)}\leq \exp{(-I^Q(\mu))};
\end{equation}
\item we have
$$\delta^Q(K):=\lim_{k\to \infty} \delta_k^Q(K)=\exp{(-V_w)};$$
\item if $\{x_j^{(k)}\}_{j=1,...,k; \ k=2,3,...}\subset K$ and 
$$\lim_{k\to \infty} |VDM_k^Q(x_1^{(k)},...,x_k^{(k)})|^{2/k(k-1)}= \exp{(-V_w)}$$
then 
$$\mu_k=\frac{1}{k}\sum_{j=1}^k\delta_{x_j^{(k)}}\to \mu_{K,Q} \ \hbox{weakly}.$$

\end{enumerate}

\end{theorem}

\section{Weighted potential theory on unbounded sets in $\C$}
We use the previous results in $\R^{3}$ and the inverse stereographic projection from the two-dimensional sphere to $\C$ to extend classical results concerning potential theory on compact subsets of $\C$ to unbounded closed sets with weakly admissible weights. Some of these results already appeared in the literature, see, e.g., \cite{HK,Sim}.


Thus let $K\subset \C$ be closed and unbounded. We consider three types of admissibility for weight functions on $K$. 
\begin{definition} \label{admit} Let $Q$ be a lower semicontinuous function on $K$ with $\{z\in K: Q(z)< \infty\}$ a nonpolar subset of $\C$ (equivalently a non log-polar subset of $\R^2$). We say $Q$ is 
\begin{enumerate}
\item {\sl weakly admissible} if  there exists $M\in(-\infty,\infty)$ such that
\begin{equation}\label{weak-adm}
\liminf_{z\in K, \ |z|\to \infty}\bigl(Q(z)-\log |z|\bigr) =M.
\end{equation}
\item {\sl admissible} if $\liminf_{z\in K, \ |z|\to +\infty}\bigl(Q(z)-\log |z|\bigr) = +\infty$.
\item {\sl strongly admissible} if for some $\epsilon >0$, there exists $R> 0$ with $Q(z) > (1+\epsilon)\log |z|$ for $z\in K$ and $|z|>R$.
\end{enumerate}
\end{definition}

\noindent Examples of weakly admissible weights arise from logarithmic potentials: if  $\mu$ is a probability measure on $\C$ such that 
$U^{\mu}$ is continuous, then $Q=-U^{\mu}$ is weakly admissible on $K=\C$. 

We assume now that $Q$ is weakly admissible. We consider the inverse stereographic projection $T:\C \cup \{\infty\}\to {\bf S}$ where ${\bf S}$ is the sphere in $\R^3$ centered in $(0,0,1/2)$ of radius $1/2$. It is defined by
\begin{equation} \label{stereo} T(z)=\left(\frac{\Re(z)}{1+|z|^{2}},\frac{\Im(z)}{1+|z|^{2}},\frac{|z|^{2}}{1+|z|^{2}}\right),\quad z\in\C\end{equation}
and $T(\infty)=P_0,$ where $P_0=(0,0,1)$ denotes the ``north pole'' of ${\bf S}$.
The map $T$ is a homeomorphism with
\begin{equation}\label{rel-dist}
|T(z)-T(u)|=\frac{|z-u|}{\sqrt{1+|z|^{2}}\sqrt{1+|u|^{2}}},\quad z,u\in\C,
\end{equation}
where $|\cdot|$ denotes the Euclidean distance. 

For $\nu$ a positive Borel measure supported on $K$, not necessarily finite, we denote by $T_{*}\nu$ its push-forward by $T$, that is, the measure on $T(K)$ such that
$$\int_{T(K)}f(x)dT_{*}\nu(x)=\int_{K}f(T(z))d\nu(z),$$
for any Borel function $f$ on $T(K)$.
Lemma 2.1 in \cite{H} shows that the map 
$$T_*:~\MM(K)\to\MM(T(K)),$$
is a homeomorphism from ${\mathcal  M}(K)$ to the subset of ${\mathcal  M}(T(K))$ of measures which put no mass at the north pole $P_0$ of ${\bf S}$. 
Here, $\MM(K)$ and $\MM(T(K))$ are endowed with the topology of weak convergence. This is the topology coming from duality with bounded, continuous functions. On $K$, it suffices to consider bounded, continuous functions $f:K\to \C$ such that $\lim_{|z|\to \infty} f(z)$ exists. This follows from the correspondence of ${\mathcal  M}(K)$ with the measures in ${\mathcal  M}(T(K))$ putting no mass at $P_0$.

When the support of a measure $\mu\in\MM(K)$ is unbounded, its potential
$$U^{\mu}(z)=\int\log\frac{1}{|z-t|}d\mu(t),\qquad z\in\C$$
is not always well-defined. However, the following lemma
holds true.
\begin{lemma}\label{pot-unbdd}
If there exists a $z_{0}\in\C$ with $U^{\mu}(z_{0})>-\infty$ then
\begin{equation}\label{cond-log-mu}
\int\log(1+|t|)d\mu(t)<\infty,
\end{equation}
which implies that $U^{\mu}(z)$ is well-defined as a function on $\C$ with values in $(-\infty,\infty]$. Moreover, $U^{\mu}(z)$ is then superharmonic and
$$-U^{\mu}(z)\leq\log(1+|z|)+\int\log(1+|t|)d\mu(t).$$
Also,
\begin{equation}\label{equiv-log-ener}
\int\log(1+|t|)d\mu(t)<\infty\quad\Longleftrightarrow\quad -\infty<I(\mu).
\end{equation}
\end{lemma}
\begin{proof}
If $1+2|z_{0}|\leq|t|$ then $1+|t|\leq2(|t|-|z_{0}|)\leq2|t-z_{0}|$, hence
\begin{equation}\label{energy2}\int_{1+2|z_{0}|\leq |t|}\log(1+|t|)d\mu(t)
\leq\log2+\int_{1+2|z_{0}|\leq |t|}\log|z_{0}-t|d\mu(t)<\infty.\end{equation}
Conversely, if (\ref{cond-log-mu}) holds then $-\infty<U^{\mu}(z)$ and $-\infty<I(\mu)$ since $|z-t|\leq
(1+|z|)(1+|t|)$.
Under assumption (\ref{cond-log-mu}), the potential $U^{\mu}(z)$ is superharmonic. This follows e.g. from the fact that 
$$U^{\mu}(z)=\int\log\frac{1+|t|}{|z-t|}d\mu(t)-\int\log(1+|t|)d\mu(t),$$
and the first integral on the right-hand side is superharmonic with respect to $z$, see \cite[Theorem 2.4.8]{R}.
The direct implication in (\ref{equiv-log-ener}) was noted above.
Conversely, if $-\infty<I(\mu)$ then $U^{\mu}(z)$ cannot be constant, equal to $-\infty$, for all $z$, so the inequality on the left of (\ref{equiv-log-ener}), which is (\ref{cond-log-mu}), follows from (\ref{energy2}).
\end{proof}
Logarithmic potentials on $\C$ and on the sphere ${\bf S}$ correspond by the relation
\begin{equation}\label{pot-C-S}
U^{\mu}(z)=U^{T_{*}\mu}(Tz)-\frac12\log(1+|z|^{2})-\frac12\int\log(1+|t|^{2})d\mu(t),\qquad z\in\C.
\end{equation}
The weighted logarithmic energy of a measure $\mu\in\MM(K)$ is defined as
\begin{equation}\label{enerK}
I^Q(\mu)=\int_{K}\int_{K}\log\frac{1}{|z-t|w(z)w(t)}d\mu(z)d\mu(t)= I(\mu) +2\int_KQd\mu,
\end{equation}
where $w=e^{-Q}$. The double integral is always well-defined. Indeed it follows from the upper semicontinuity of $w$ and (\ref{weak-adm}) that the integrand is bounded below. On the contrary, the second expression has a meaning only if $I(\mu)>-\infty$ which is not necessarily true. Another equivalent way to define $I^{Q}(\mu)$, which is always valid, is by using the map $T$ as was done in \cite{HK}. Here, one identifies  $I^{Q}(\mu)$ with the weighted logarithmic energy of the measure $T_{*}\mu\in\MM(T(K))$, 
\begin{equation}\label{enerTK}
I^{\tilde Q}(T_* \mu)= \int_{T(K)}\int_{T(K)}\log {1\over |x-y|}dT_* \mu(x) d T_* \mu(y) + 2\int_{T(K)} \tilde Q(x)dT_*\mu(x), 
\end{equation}
where
\begin{equation}\label{tildeq} \tilde Q(T(z))=Q(z) -{1\over 2} \log (1+|z|^2).\end{equation}
To define $\tilde Q$ on the whole of $T(K)$ we set $\tilde Q(P_0)=M$, so that $\tilde Q$ becomes lower semicontinuous, and we get a correspondence between {\it weakly admissible weights on the closed set $K$ in $\C$} and {\it admissible weights on the compact set $T(K)\subset {\bf S}$ in $\R^{3}$}.
\begin{lemma}
A closed subset $K\subset\C$ is polar if and only if $T(K)\subset {\bf S}$ is log-polar.
\end{lemma}
\begin{proof}
We have
$$\int_{T(K)}\int_{T(K)}\log {1\over |x-y|}dT_* \mu(x) d T_* \mu(y)=
\int_{K}\int_K\log\frac{1}{|z-t|}d\mu(z)d\mu(t)+\int_{K}\log(1+|t|^{2})d\mu(t).$$
Recall that a closed subset $K\subset\C$ is polar if $K_{r}=K\cap B(0,r)$ is polar for all $r>0$. Thus, if $K$ is nonpolar, there exists $r>0$ with $K_{r}$ nonpolar, that is, there is a measure $\mu_{r}$ supported on $K_{r}$ of finite energy. By the above equality, $T_{*}\mu_{r}$ is a measure on $T(K_{r})\subset T(K)$ of finite energy, so $T(K)$ is not log-polar.  Conversely, if $K\subset\C$ is polar, for any finite measure $\mu$ of compact support in $K$ we have $I(\mu)=\infty$ in $\C$ (cf. \cite{R}) and thus $I(T_{*}\mu)=\infty$ in ${\bf S}$. Since any measure on $T(K)$ charging the north pole $P_{0}$ has infinite energy, it follows that $T(K)$ is log-polar. 
\end{proof}
Theorem \ref{frost} asserts the existence and uniqueness of a weighted energy minimizing measure on a non log-polar compact subset of $\R^{n}$ with an admissible weight. Obviously, this minimizing measure does not charge any point of the set, in particular the north pole $P_0$ if it belongs to the set. Hence the above correspondence implies the following.
\begin{theorem} \label{Frost-C}
Let $K$ be a nonpolar closed subset of $\C$ and $Q$ a weakly admissible weight on $K$. Then,
\begin{enumerate}[noitemsep,nolistsep]
\item $V_w:=\inf_{\mu\in \mathcal M(K)}I^Q(\mu)$ is finite;
\item there exists a unique weighted equilibrium measure $\mu_{K,Q} \in \mathcal M(K)$ with $I^Q(\mu_{K,Q})=V_w$ and the logarithmic energy $I(\mu_{K,Q})$ is finite (hence $-\Delta U^{\mu_{K,Q}}=2\pi\mu_{K,Q}$);
\item the support $S_w:=$supp$(\mu_{K,Q})$ is contained in $\{x\in K: Q(z) <\infty\}$ and $S_w$ is not polar;
\item Let $F_w:=V_w -\int_K Q(z)d\mu_{K,Q}(z)$ denote the (finite) Robin constant. Then
\begin{align*}
U^{\mu_{K,Q}}(z)+Q(z) & \geq F_w\text{ on }K\setminus P\text{ where }P
\text{ is polar (possibly empty);}\\
U^{\mu_{K,Q}}(z)+Q(z) & \leq F_w\text{ for all }z\in S_w.
\end{align*}
\end{enumerate}
\end{theorem}
\begin{proof}
The above assertions correspond by the map $T$ to the similar assertions from Theorem \ref{frost} applied with a non log-polar compact subset of the sphere ${\bf S}$. Note that
$$T_{*}\mu_{K,Q}=\mu_{T(K),\tilde Q},\qquad V_{w}=V_{\tilde w}, \qquad F_{w}=F_{\tilde w}-
\frac12\int_{K}\log(1+|t|^{2})d\mu_{K,Q}(t),$$
where we have set $\tilde w:=e^{-\tilde Q}$.
The fact that $I(\mu_{K,Q})<\infty$ follows from 
$$I(\mu_{K,Q})=I^{Q}(\mu_{K,Q})-2\int_{K}Qd\mu_{K,Q}
=I^{\tilde Q}(T_{*}\mu_{K,Q})-2\int_{K}Qd\mu_{K,Q},$$
where we know that $I^{\tilde Q}(T_{*}\mu_{K,Q})=V_{\tilde w}$ is finite and $Q$ is bounded below. If $S_{w}$ is compact, it is clear that the other inequality $-\infty<I(\mu_{K,Q})$ is satisfied. If $S_{w}$ is not compact, 
we may use
$$\iint\log|z-t|d\mu_{K,Q}(z)d\mu_{K,Q}(t)\leq2\int\log(1+|t|)d\mu_{K,Q}(t),$$
so to verify $-\infty<I(\mu_{K,Q})$ it suffices to show that 
$$\int_{K}\log(1+|t|)d\mu_{K,Q}(t)<\infty,$$
which holds true by Lemma \ref{pot-unbdd} since the equilibrium potential satisfies $U^{\mu_{K,Q}}>-\infty$. 
\end{proof}
In particular, if $\mu$ is a probability measure on $\C$ such that 
$U^{\mu}$ is continuous, taking $Q=-U^{\mu}$ on $K=\C$ we have $\mu=\mu_{K,Q}$ so that, in general, $\mu_{K,Q}$ need not have compact support. As specific examples, if $K =\C$ and $Q(z)=\frac{1}{2}\log (1+|z|^2)$, then $d\mu_{K,Q}=
\pi^{-1}(1+|z|^2)^{-2}dm(z)$ where $dm$ is Lebesgue measure, cf., Example 1.4 of \cite{H}. If $K=\R$ and $Q(x)=\frac{1}{2}\log (1+x^2)$, then $d\mu_{K,Q}=\pi^{-1}(1+x^2)^{-1}dx$,  cf., Example 1.3 of \cite{H}. We mention that in \cite{HK}, existence and uniqueness of a minimizing measure was proven in the more general context of weakly admissible vector equilibrium problems.

Let $L(\C)$ be the set of all subharmonic functions $u$ on $\C$ with the property that $$u(z)-\log |z|\text{ is bounded above as }|z| \to \infty.$$ 
We will need the following version of the domination principle, see \cite[Corollary A.2]{BL}.
\begin{proposition}
Let $u,v\in L(\C)$ with $u(z)-v(z)$ bounded above as $|z|\to \infty$ and suppose $I(\Delta v)<\infty$. If $u\leq v$ a.e.-$\Delta v$, then 
$u\leq v$ on $\C$.
\end{proposition}
\noindent Here, $\Delta v$ need not have compact support.

We can now state a weighted version of the Bernstein-Walsh lemma with a weakly admissible weight (see \cite[Theorem III.2.1]{ST} for the case of an admissible weight). This will be used in section 8 to get a version for appropriate polynomials on the sphere (Theorem \ref{BW-S}). We let $\mathcal P_k(\C)$ denote the complex-valued polynomials of a complex variable of degree at most $k$. 
\begin{theorem}\label{BW-C}
Let $K$ be a closed nonpolar subset of $\C$ and $Q$ a weakly admissible weight on $K$. If $p_{k}\in \mathcal P_k(\C)$ and
$$|p_{k}(z)e^{-kQ(z)}|\leq M\qquad\text{ for q.e. }z\in S_{w},$$
then 
$$|p_{k}(z)|\leq M\exp(k(-U^{\mu_{K,Q}}(z)+F_{w})),\qquad z\in\C,$$
and
$$|p_{k}(z)e^{-kQ(z)}|\leq M,\qquad \text{for q.e. }z\in K.$$
\end{theorem}
\begin{proof}
The function $g:=\log(|p_{k}|/M)/k$ belongs to $L(\C)$ and 
$$g(z)\leq Q(z)\leq -U^{\mu_{K,Q}}(z)+F_{w}\quad\text{for q.e. }z\in S_{w}.$$
By Lemma \ref{pot-unbdd}, $-U^{\mu_{K,Q}}+F_{w}$ also belongs to $L(\C)$ and 
$-\Delta U^{\mu_{K,Q}}=2\pi\mu_{K,Q}$ is of finite energy. Hence, by the above domination principle, 
$$g(z)\leq -U^{\mu_{K,Q}}(z)+F_{w},\qquad z\in\C,$$
which, together with the first inequality in item 4. of Theorem \ref{Frost-C}, proves our contention.
\end{proof}
We proceed with properties of the weighted Vandermonde. We have the relation
$$|VDM_k^Q(z_1,...,z_k)|=|VDM_k^{\tilde Q}(T(z_1),...,T(z_k))|,$$
from which it follows that the assertions of Theorem \ref{sec3} about the Vandermonde can be carried over to $\C$. Since the result may be of interest on its own, we state it as a theorem.
\begin{theorem} \label{afpinc}
Let $K$ be a closed nonpolar subset of $\C$ and $Q$ a weakly admissible weight on $K$. The $k$-th weighted diameters $\delta_{k}^{Q}(K)$, $k\geq 2$, are finite and
\begin{enumerate}
\item if $\{\mu_k=\frac{1}{k}\sum_{j=1}^k\delta_{x_j^{(k)}}\}\subset \mathcal M(K)$ converge weakly to $\mu\in \mathcal M(K)$, then 
\begin{equation}\label{upboundVDM-C}
\limsup_{k\to \infty} |VDM_k^Q(x_1^{(k)},...,x_k^{(k)})|^{2/k(k-1)}\leq \exp{(-I^Q(\mu))};
\end{equation}
\item we have
$\delta^Q(K):=\lim_{k\to \infty} \delta_k^Q(K)=\exp{(-V_w)};$
\item if $\{x_j^{(k)}\}_{j=1,...,k; \ k=2,3,...}\subset K$ and 
$$\lim_{k\to \infty} |VDM_k^Q(x_1^{(k)},...,x_k^{(k)})|^{2/k(k-1)}= \exp{(-V_w)}$$
then 
$$\mu_k=\frac{1}{k}\sum_{j=1}^k\delta_{x_j^{(k)}}\to \mu_{K,Q} \ \hbox{weakly}.$$
\end{enumerate}
\end{theorem}

\begin{remark} The Frostman-type result in 4. of Theorem \ref{Frost-C} and 2. of Theorem \ref{afpinc} (as well as 3. in the special case of arrays of weighted Fekete points) have also been proved in \cite{Bet}.
\end{remark}

\section{Bernstein-Markov properties in $\R^n$}\label{Sec-BM}
In sections 4-8, we return to the setting of compact sets in $\R^n$. In particular, admissible weights will be in the sense of Definition \ref{realad}. For $k=1,2,...$, let $\mathcal P_k=\mathcal P_k^{(n)}$ denote the \emph{real} polynomials  in $n$ real variables $x=(x_1,...,x_n)$ of degree at most $k$ and $\mathcal P_k(\C)$ denote the \emph{complex} holomorphic polynomials  in $n$ complex variables $z=(z_1,...,z_n)$ of degree at most $k$. Given a compact set $K\subset \C^n$ and a positive measure $\nu$ on $K$, we say that $(K,\nu)$ satisfies the Bernstein-Markov property (or $\nu$ is a Bernstein-Markov measure for $K$) if for all $p_k\in \mathcal P_k(\C)$, 	
$$||p_k||_K:=\sup_{z\in K} |p_k(z)|\leq  M_k||p_k||_{L^2(\nu)}  \ \hbox{with} \ \limsup_{k\to \infty} M_k^{1/k} =1.$$
It was shown in \cite{PELD} that any compact set in $\C^n$ admits a Bernstein-Markov measure for complex holomorphic polynomials; indeed, the following stronger statement is true.

\begin{proposition} [\cite{PELD}]  \label{allbm} Let $K\subset \R^n$. There exists a measure $\nu \in \mathcal M(K)$ such that for all complex-valued polynomials $p$ of degree at most $k$ in the (real) coordinates $x=(x_1,...,x_n)$ we have
$$||p||_K\leq M_k ||p||_{L^2(\nu)}$$
where $\limsup_{k\to \infty}M_k^{1/k}=1$.

\end{proposition} 

\noindent For a compact set $K\subset \R^n$ and a positive measure $\nu$ on $K$, we will say that $(K,\nu)$ satisfies the Bernstein-Markov property if for all $p_k\in \mathcal P_k$, 	
$$||p_k||_K:=\sup_{x\in K} |p_k(x)|\leq  M_k||p_k||_{L^2(\nu)}  \ \hbox{with} \ \limsup_{k\to \infty} M_k^{1/k} =1.$$
More generally, for $K\subset \R^n$ compact, $Q\in \AA(K)$, and $\nu$ a measure on $K$, we say that the triple $(K,\nu,Q)$ satisfies the weighted Bernstein-Markov property if for all $p_k\in \mathcal P_k$, 
$$||e^{-kQ}p_k||_K \leq M_k ||e^{-kQ}p_k||_{L^2(\nu)} \ \hbox{with} \ \limsup_{k\to \infty} M_k^{1/k} =1.$$ 
We have the analogous notion of weighted Bernstein-Markov property for $p_k\in \mathcal P_k(\C)$ if $K\subset \C^n$ and $Q\in \AA(K)$.

\begin{remark}\label{lp} These properties can be stated with $L^p$ norms for any $0<p< \infty$. The proof of Theorem 3.4.3 in \cite{StTo} in $\C$ that if $(K,\nu)$ satisfies an (weighted) $L^p-$Bernstein-Markov property  for $\mathcal P_k(\C)$ for some $0<p<\infty$ then $(K,\nu)$ satisfies an (weighted) $L^p-$Bernstein-Markov property for all $0<p<\infty$ just uses H\"older's inequality and remains valid in our setting. 

\end{remark}

Now another very important observation: Theorem 3.2 of \cite{bloom} works -- indeed, is even stated -- in $\R^n$ for any $n\geq 2$:

\begin{theorem} \label{blooms} Given $K\subset \R^n$ compact, and $Q$ a {\it continuous} weight, if $\nu$ is a finite measure on $K$ such that $(K,\nu)$ satisfies a Bernstein-Markov property, then the triple $(K,\nu,Q)$ satisfies a weighted Bernstein-Markov property.

\end{theorem} 

\begin{definition} Given $K\subset \R^n$ compact, a finite measure $\nu$ on $K$ is called a {\it strong Bernstein-Markov measure} for $K$ if for any continuous weight $Q$ on $K$, the triple $(K,\nu,Q)$ satisfies a weighted Bernstein-Markov property. We have the analogous notion if $K\subset \C^n$ using $p_k\in \mathcal P_k(\C)$.

\end{definition}

\begin{remark}\label{strongbm}  Combining Proposition \ref{allbm} and Theorem \ref{blooms} we see that {\it any compact set $K$ in $\R^n$ admits a strong Bernstein-Markov measure; and any Bernstein-Markov measure on $K$ is automatically a strong Bernstein-Markov measure for $K$}. This latter equivalence is not necessarily true in the complex setting. For $K\subset \C$, there are well-known sufficient mass-density conditions on a measure $\nu$ on $K$ so that $(K,\nu)$ satisfies a Bernstein-Markov property for $\mathcal P_k(\C)$ \cite{StTo}. In particular, Lebesgue measure on an interval or Lebesgue planar measure on a compact set in $\C$ having $C^1$ boundary satisfy the Bernstein-Markov property. We remark that if $\C\setminus K$ is regular for the Dirichlet problem, the condition that $(K,\nu)$ satisfies a Bernstein-Markov property for $\mathcal P_k(\C)$ is equivalent to the condition that $\nu$ be a regular measure; i.e., $\nu \in {\bf Reg}$ in the terminology of \cite{StTo}. We refer to this book for more details.

Furthermore, for every compact set $K$ in $\R^n$ there exist discrete measures which satisfy  the (strong) Bernstein-Markov property \cite{PELD}. If one considers $K\subset \R^n\subset \C^n$, there are sufficient mass-density conditions on a measure $\nu$ on $K$ so that $(K,\nu)$ satisfies a Bernstein-Markov property for polynomials on $\C^n$ and hence on $\R^n$. For more on this, cf., \cite{BLmass} and \cite{PELD}.
\end{remark}

 \begin{remark}\label{bmprop} Let $K\subset \R^n$ be compact and not log-polar and let $v\in \mathcal A(K)$. If $\alpha$ is a finite measure on $K$ such that $(K,\alpha,v)$ satisfies a weighted Bernstein-Markov property, then 
\begin{enumerate}
\item $(K,c\alpha,v)$ satisfies a weighted Bernstein-Markov property for any $0< c <\infty$ and 
\item $(K,\alpha+\beta,v)$ satisfies a weighted Bernstein-Markov property for any finite measure $\beta$ on $K$. 
\end{enumerate}

\end{remark}

The importance of a (weighted) Bernstein-Markov property is the following consequence on the asymptotic behavior of the (weighted) {\it $L^{2}$ normalization constants} defined by
\begin{equation}\label{L2}
Z^Q_k=Z^{Q}_k(K,\nu):=\int_{K^{k}}|VDM_k^Q({\bf X_k})|^2d\nu({\bf X_k}),
\end{equation}
where ${\bf X_k}:=(x_1,...,x_k)\in K^{k}$ and $\nu$ is a finite positive measure on $K$.
\begin{remark}\label{L2-Z}
The quantity $Z_{k}^{Q}$ appears as the normalization constant in the law of eigenvalues of random matrix models. It is also referred to as the partition function in the theory of Coulomb gases. See Section \ref{appli} for more details on the link between these notions. 
\end{remark}
\begin{proposition} \label{weightedtd} 
Given $K\subset \R^n$ compact and not log-polar, $Q\in \AA(K)$, and $\nu$ a finite measure on $K$ such that $(K,\nu,Q)$ satisfies a weighted Bernstein-Markov property, we have
$$\lim_{k\to \infty} (Z_k^Q)^{1/k(k-1)}=\exp{(-V_w)}=\delta^{Q}(K).$$
\end{proposition}

\begin{proof} We clearly have $\limsup_{k\to \infty} (Z_k^Q)^{1/k(k-1)}\leq \exp{(-V_w)}$ from 2. of Theorem \ref{sec3}. For the reverse inequality with $\liminf$, note that 
$$|VDM_k(x_1,...,x_k)|^2 =\prod_{i<j} |x_i-x_j|^2=\prod_{i<j} \bigl(\sum_{l=1}^n(x_{i,l}-x_{j,l})^2\bigr)$$
(where we write $x_i =(x_{i,1},...,x_{i,n})$)  
is a polynomial of degree $k(k-1)$ in the $nk$ real coordinates $\{x_{i,l}\}_{l=1,...,n; \ i=1,...,k}$. 
Now if ${\bf F_k}=(f_1,...,f_k)$ is a set of weighted Fekete points of order $k$ for $K,Q$, then 
$$p(x_1):=|VDM_k(x_1,f_2,...,f_k)|^2\prod_{j=2}^ke^{-2(k-1)Q(f_j)}$$
is a (nonnegative) polynomial of degree $2(k-1)$ in (the coordinates of) $x_1$. By definition of weighted Fekete points, for any $x_1\in K$,
$$p(x_1) e^{-2(k-1)Q(x_1)} \leq \max_{x\in K} p(x)e^{-2(k-1)Q(x)}=p(f_1) e^{-2(k-1)Q(f_1)}$$
since this right-hand-side is precisely $|VDM_k^Q({\bf F_k})|^2$. By the weighted Bernstein-Markov property using $L^1$ norm instead of $L^2$ (see Remark \ref{lp}), 
$$|VDM_k^Q({\bf F_k})|^2\leq M_{2(k-1)} \int_K |VDM_k(x_1,f_2,...,f_k)|^2\cdot e^{-2(k-1)Q(x_1)}\cdot \prod_{j=2}^ke^{-2(k-1)Q(f_j)}d\nu(x_1).$$
Now for each fixed $x_1\in K$, we consider
$$p_2(x_2):= |VDM_k(x_1,x_2,f_3...,f_k)|^2\cdot e^{-2(k-1)Q(x_1)} \cdot \prod_{j=3}^ke^{-2(k-1)Q(f_j)}$$
which is a (nonnegative) polynomial of degree $2(k-1)$ in (the coordinates of) $x_2$. Then
$$p_2(f_2) e^{-2(k-1)Q(f_2)}\leq  \max_{x\in K} p_2(x) e^{-2(k-1)Q(x)}.$$
The left-hand-side is 
$$|VDM_k(x_1,f_2,f_3...,f_k)|^2\cdot e^{-2(k-1)Q(x_1)} \cdot \prod_{j=2}^ke^{-2(k-1)Q(f_j)}.$$
The right-hand-side, by the weighted Bernstein-Markov property, is bounded above by
$$M_{2(k-1)} \int_K |VDM_k(x_1,x_2,f_3,...,f_k)|^2\cdot e^{-2(k-1)Q(x_1)}\cdot\prod_{j=3}^ke^{-2(k-1)Q(f_j)}e^{-2(k-1)Q(x_2)}d\nu(x_2).$$
Plugging these into our first estimate, we have
$$|VDM_k^Q({\bf F_k})|^2\leq M_{2(k-1)} \int_K |VDM_k(x_1,f_2,...,f_k)|^2\cdot e^{-2(k-1)Q(x_1)}\cdot \prod_{j=2}^ke^{-2(k-1)Q(f_j)}d\nu(x_1)$$
$$\leq M_{2(k-1)} \int_K\bigl[  M_{2(k-1)} \int_K |VDM_k(x_1,x_2,f_3,...,f_k)|^2\cdot e^{-2(k-1)Q(x_1)}\cdot$$
$$\prod_{j=3}^ke^{-2(k-1)Q(f_j)}e^{-2(k-1)Q(x_2)}d\nu(x_2)\bigr]d\nu(x_1)$$
$$=(M_{2(k-1)} )^2\int_{K^2} |VDM_k(x_1,x_2,f_3,...,f_k)|^2\prod_{j=3}^ke^{-2(k-1)Q(f_j)}\prod_{j=1}^2e^{-2(k-1)Q(x_j)}d\nu(x_2)d\nu(x_1).$$
Continuing the process and using $M_{2(k-1)}^{1/2k}\to 1$ gives the result.

\end{proof}

Given {\it any} $Q\in \AA(K)$, we can always find a finite measure satisfying the important conclusion of Proposition \ref{weightedtd}.

\begin{proposition}\label{7to4} Let $K\subset \R^n$ be compact and not log-polar and let $Q\in \mathcal A(K)$. Then there exists a finite measure $\mu$ on $K$ such that 
\begin{equation}\label{in7to4}\lim_{k\to \infty} (Z_k^Q(K,\mu))^{1/k(k-1)}=\exp{(-V_w)}=\delta^{Q}(K).\end{equation}
We can even construct $\mu$ so that, in addition, $\mu$ is a (strong) Bernstein-Markov measure for $K$.
\end{proposition} 

\begin{proof} Consider the weighted equilibrium measure $\mu_{K,Q}$. We have $V_w=I(\mu_{K,Q})<\infty$. By Lusin's continuity theorem, for every integer $m> 1$, there exists a compact subset $K_{m}$ of $K$ such that $\mu_{K,Q}(K\setminus K_{m})\leq 1/m$ and $Q$ (considered as a function on $K_{m}$ only) is continuous on $K_{m}$. We may assume that each $K_m$ is not log-polar and that the sets $K_{m}$ are increasing as $m$ tends to infinity.  Let $\mu_m\in \mathcal M(K_m)$ be a (strong) Bernstein-Markov measure for $K_m$. We claim that $\mu=\sum_{m=1}^{\infty} \frac{1}{2^m}\mu_m$ satisfies (\ref{in7to4}).

Since $\mu_m\in \mathcal M(K_m)$ is a (strong) Bernstein-Markov measure for $K_m$ and $Q_m:=Q|_{K_m}$ is continuous, it also follows from 1. of Remark \ref{bmprop} that $(K_m,Q_m,\frac{1}{2^m}\mu_m)$ satisfies a weighted Bernstein-Markov property. 

Since $Z_k^Q(K,\mu)\leq \max_{{\bf x}\in K^k}|VDM^Q({\bf x})|^2\mu(K)^k$, we have
$$\limsup_{k\to \infty} (Z_k^Q(K,\mu))^{1/k(k-1)}\leq \delta^{Q}(K).$$
To show
$$\liminf_{k\to \infty} (Z_k^Q(K,\mu))^{1/k(k-1)}\geq \delta^{Q}(K),$$
let $\lambda_m:= \mu_{K,Q}(K_m)$ so that $\lambda_m\uparrow 1$. Letting 
$\theta_m:=\frac{1}{\lambda_m} \mu_{K,Q}|_{K_m}\in \mathcal M(K_m)$, we have 
$$I^{Q_m}(\theta_m)\geq I^{Q_m}(\mu_{K_m,Q_m}).$$
Since $(K_m,Q_m,\frac{1}{2^m}\mu_m)$ satisfies a weighted Bernstein-Markov property, 
$$\exp{(- I^{Q_m}(\mu_{K_m,Q_m}))}=\lim_{k\to \infty}(Z_k^{Q_m}(K_m,\frac{1}{2^m}\mu_m))^{1/k(k-1)}.$$
Clearly 
$$Z_k^Q(K,\mu)\geq Z_k^{Q_m}(K_m,\mu|_{K_m})\geq Z_k^{Q_m}(K_m,\frac{1}{2^m}\mu_m).$$
Thus
$$\liminf_{k\to \infty} (Z_k^Q(K,\mu))^{1/k(k-1)}\geq \liminf_{k\to \infty}  (Z_k^{Q_m}(K_m,\frac{1}{2^m}\mu_m))^{1/k(k-1)}$$
$$=\exp{(- I^{Q_m}(\mu_{K_m,Q_m}))}\geq \exp{(-I^{Q_m}(\theta_m))}.$$
By monotone convergence we have
$$\lim_{m\to \infty} I^{Q_m}(\theta_m)= I^Q(\mu_{K,Q})$$
so that 
$$\liminf_{k\to \infty} (Z_k^Q(K,\mu))^{1/k(k-1)}\geq \exp{-I^Q(\mu_{K,Q})}=\delta^{Q}(K),$$
as desired.

For the second part, let $\nu$ be a (strong) Bernstein-Markov measure for $K$ and define 
$$\mu:= \sum_{m=1}^{\infty} 2^{-m}\mu_m + \nu.$$ 
The fact that $\mu$ is a (strong) Bernstein-Markov measure for $K$ follows from the fact that $\nu$ is a (strong) Bernstein-Markov  measure for $K$ and 2. of Remark \ref{bmprop}. Finally, $\mu$ satisfies (\ref{in7to4}) from the previous part applied to $\sum_{m=1}^{\infty} 2^{-m}\mu_m$ and the obvious inequality $Z_k^Q(K,\mu) \geq Z_k^Q(K,\sum_{m=1}^{\infty} 2^{-m}\mu_m)$.
\end{proof}

\begin{example} If $\mu$ is a (strong) Bernstein-Markov measure for $K$ and the set of points of discontinuity of $Q\in \mathcal A(K)$ is of $\mu$ measure zero, then (\ref{in7to4}) holds for $\mu$. As a simple but illustrative example, let $K=[-1,1]\subset \R$ and take 
$$Q(x)=0 \ \hbox{at all}  \ x\in [-1,1]\setminus \{0\}; \ Q(0)=-1.$$
It is easy to see that Lebesgue measure $d\mu$ on $[-1,1]$ satisfies (\ref{in7to4}) but $(K,Q,\mu)$ does not satisfy the weighted Bernstein-Markov property. On the other hand, $(K,Q,\mu+\delta_0)$ does satisfy the weighted Bernstein-Markov property.\end{example}

For $Q\in \mathcal A(K)$ and $\nu$ a finite measure on $K$, 
we define a probability measure $Prob_k$ on $K^{k}$: for a Borel set $A\subset K^{k}$,
\begin{equation}\label{probk}Prob_k(A):=\frac{1}{Z^Q_k}\cdot \int_A  |VDM_k^Q({\bf X_k})|^2  d\nu({\bf X_k}).
\end{equation}
Directly from Proposition \ref{7to4} and (\ref{probk}) we obtain the following estimate.

\begin{corollary} \label{johansson} 
Let $Q\in \mathcal A(K)$ and $\nu$ a finite measure on $K$ satisfying 
$$\lim_{k\to \infty} (Z_k^Q(K,\nu))^{1/k(k-1)}=\exp{(-V_w)}=\delta^{Q}(K).$$
Given $\eta >0$, define
 \begin{equation}\label{aketa}
 A_{k,\eta}:=\{{\bf X_k}\in K^{k}: |VDM_k^Q({\bf X_k})|^2 \geq 
 ( \delta^Q(K) -\eta)^{k(k-1)}\}.
 \end{equation}
Then there exists $k^*=k^*(\eta)$ such that for all $k>k^*$, 
$$Prob_k(K^{k}\setminus A_{k,\eta})\leq \Big(1-\frac{\eta}{2 \exp{(-V_w)}}\Big)^{k(k-1)}\nu(K^{k}). 
$$
\end{corollary}	
	
	We get the induced product probability measure ${\bf P}$ on the space of arrays on $K$, 
	$$\chi:=\{X=\{{\bf X_{k}}\in K^{k}\}_{k\geq 1}\},$$ 
	namely,
	$$(\chi,{\bf P}):=\prod_{k=1}^{\infty}(K^{k},Prob_k).$$
As an immediate consequence of the Borel-Cantelli lemma and 3. of Theorem \ref{sec3}, we obtain: 

\begin{corollary}\label{416} Let $Q\in \mathcal A(K)$ and $\nu$ a finite measure on $K$ satisfying $$\lim_{k\to \infty} (Z_k^Q(K,\nu))^{1/k(k-1)}=\exp{(-V_w)}=\delta^{Q}(K).$$  
For ${\bf P}$-a.e. array $X=\{x_j^{(k)}\}_{j=1,...,k; \ k=2,3,...}\in \chi$, 
$$
\frac{1}{k}\sum_{j=1}^k\delta_{x_j^{(k)}}\to \mu_{K,Q} \ \hbox{weakly as }
k\to\infty.$$
\end{corollary}

\section{Approximation of probability measures}
For the proof  of a large deviation principle (LDP) in $\R^n$, as in \cite{VELD}, we will need to approach general measures in $\MM(K)$ by weighted equilibrium measures. For that, we consider equilibrium problems with weights that are the negatives of potentials. We first verify that the natural candidate solution to such a problem is, indeed, the true solution.
\begin{lemma}\label{lem-non-adm}
Let $\mu\in\MM(K)$, $K\subset \R^n$ compact, $I(\mu)<\infty$. Consider the possibly non-admissible weight $u:=-U^{\mu}$ on $K$. The weighted minimal energy on $K$ is obtained with the measure $\mu$, that is
$$\forall\nu\in\MM(K),\quad I(\mu)+2\int ud\mu\leq I(\nu)+2\int ud\nu,$$
with equality if and only if $\nu=\mu$. 
\end{lemma}
\begin{proof}
We may assume that $I(\nu)<\infty$. The inequality may be rewritten as
$$0\leq I(\nu)-2I(\mu,\nu)+I(\mu)=I(\nu-\mu),$$
which is true, and, moreover, the energy $I(\nu-\mu)$ can vanish only when $\nu=\mu$, see item 1. of Theorem \ref{basics}.
\end{proof}
The following two approximation results are analogous to \cite[Lemma I.6.10]{ST}.
\begin{lemma}\label{lemma-scal}
Let $K\subset \R^n$ be compact and non log-polar and let $\mu\in\MM(K)$.
 Let $Q\in\AA(K)$ be finite $\mu$-almost everywhere. 
There exist an increasing sequence of compact sets $K_m$ in $K$  and a sequence of measures $\mu_m\in \MM(K_m)$ satisfying
\begin{enumerate}
\item the measures $\mu_m$ tend weakly to $\mu$, as $m\to\infty$; 
\item the functions $Q_{|K_m}\in C(K_m)$ and $\int Qd\mu_m$ tend to $\int Qd\mu$ as $m\to\infty$;
\item the energies $I(\mu_m)$ tend to $I(\mu)$ as $m\to\infty$.
\end{enumerate}
\end{lemma}
\begin{proof} 
By Lusin's continuity theorem, for every integer $m\geq 1$, there exists a compact subset $K_m$ of $K$ such that $\mu(K\setminus K_m)\leq 1/m$ and $Q$ (considered as a function on $K_m$ only) is continuous on $K_m$. We may assume that $K_m$ is increasing as $m$ tends to infinity. Then, the measures $\tilde\mu_m:=\mu_{|K_m}$ are increasing and tend weakly to $\mu$. 
Since $Q$ is bounded below on $K$, the monotone convergence theorem tells us that
$$\int Qd\tilde\mu_m=\int Q_{|K_m}d\mu\to\int Qd\mu,\quad\text{as }m\to\infty.$$
Denoting as usual by $\log^{+}$ and $\log^{-}$ the positive and negative parts of the $\log$ function, we have, as $m\to\infty$,
$$\chi_{m}(z,t)\log^{+}|z-t|\uparrow\log^{+}|z-t|\quad\text{and}\quad
\chi_{m}(z,t)\log^{-}|z-t|\uparrow\log^{-}|z-t|,$$
$(\mu\times\mu)$-almost everywhere on $K\times K$ where $\chi_{m}(z,t)$ is the characteristic function of $K_{m}\times K_{m}$ and we agree that the left-hand sides vanish when $z=t\notin K_{m}$. By monotone convergence, we obtain
$$I(\tilde\mu_m)\to I(\mu),\quad\text{as }m\to\infty.$$
Finally, defining $\mu_m:= \tilde \mu_m/\mu(K_m)$ gives the result.
\end{proof}

\begin{corollary} \label{corapprox} 
Let $K\subset \R^n$ be compact and non log-polar and let $\mu\in\MM(K)$ with $I(\mu)<\infty$. Let $K_m$ be the sequence of increasing compact sets in $K$ and $\mu_m$ the sequence of measures in $\MM(K_m)$ given by Lemma \ref{lemma-scal} with $Q=U^{\mu}$. There exist a sequence of continuous functions $Q_m$ on $K$ such that 
\begin{enumerate}
\item the measures $\mu_m$ tend weakly to $\mu$ and the energies $I(\mu_m)$ tend to $I(\mu)$, as $m\to\infty$; 
\item  the measures $\mu_m$ are equal to the weighted equilibrium measures $\mu_{K,Q_m}$.
\end{enumerate}
\end{corollary}
\begin{proof} First, note that $U^{\mu}\in\AA(K)$ and is finite $\mu$-almost everywhere since $I(\mu)<\infty$, so that Lemma \ref{lemma-scal} applies with $Q=U^{\mu}$. Now we define 
$$Q_m:=-U^{\mu_m}|_K=-\mu(K_m)^{-1}U^{\tilde\mu_m}|_K.$$ Since $Q=U^{\mu}$ is continuous on $K_m$, it follows that $Q_m$ is continuous on $K_m$. By the continuity principle for logarithmic potentials (4. of Theorem \ref{basics}), $-U^{\mu_m}$ is continuous on $\R^n$ and hence $Q_m$ is continuous 
on $K$. Items 1. and 2. follow from Lemmas \ref{lemma-scal} and \ref{lem-non-adm} respectively.
\end{proof}

\section{The $J^{Q}$ functionals on $\R^n$}\label{J-funct}
 In this section, we introduce and establish the main properties of the weighted
$L^2$ functionals 
$\overline J^{Q},\underline J^{Q}$ as well as the relation with the weighted energy $I^Q$.  Our goal is to establish an LDP in the next section.

Fix a compact set $K$ in $\R^{n}$, 
a measure
 $\nu$ in ${\mathcal M}(K)$ and $Q\in\AA(K)$.
We recall that $\MM(K)$ endowed with the weak topology is a Polish space, i.e., a separable complete metrizable space.
Given $G\subset {\mathcal M}(K)$, for each $k=1,2,...$ we let 
\begin{equation}\label{nbhddef}
\tilde G_k:=\{{\bf a} =(a_{1},...,a_{k})\in K^{k},~
\frac{1}{k}\sum_{j=1}^{{k}} \delta_{a_{j}}\in G\},
\end{equation}
and set
\begin{equation}\label{jkqmu}
J^Q_k(G):=\Big[\int_{\tilde G_{k}}|VDM^Q_k({\bf a})|^{2}d\nu ({\bf a})\Big]
^{1/k(k-1)}.
\end{equation}
\begin{definition} \label{jwmuq} For $\mu \in \mathcal M(K)$ we define
$$\overline J^Q(\mu):=\inf_{G \ni \mu} \overline J^Q(G) \ \hbox{where} \ \overline J^Q(G):=\limsup_{k\to \infty} J^Q_k(G);$$
$$\underline J^Q(\mu):=\inf_{G \ni \mu} \underline J^Q(G) \ \hbox{where} \ \underline J^Q(G):=\liminf_{k\to \infty} J^Q_k(G);$$
\end{definition}

\noindent Here the infima are taken over all neighborhoods $G$ of the measure $\mu$ in ${\mathcal M}(K)$. 
Note that
, a priori, $\overline J^{Q},\underline J^{Q}$ depend on $\nu$. For the unweighted case $Q=0$, we simply write $\overline J$ and $\underline J$.
 \begin{lemma}
The functionals $\underline J(\mu)$, $\overline J(\mu)$, $\underline J^Q(\mu)$, $\overline J^Q(\mu)$, are upper semicontinuous  on ${\mathcal M}(K)$ in the weak topology. 
\end{lemma}
\begin{proof}
The proof is similar to the one of \cite[Lemma 3.1]{BL}.
\end{proof}
\begin{lemma}\label{lem-J-JQ}
The following properties hold (and with the $\underline J,\underline J^Q$ 
functionals as well): 
\begin{enumerate}
\item $\overline J^Q(\mu)\leq e^{-I^Q(\mu)}$ for $Q\in \AA(K)$;
\item $ 
\overline J^{Q}(\mu)\leq\overline J(\mu)\cdot e^{-2\int_K Qd\mu}$  for $Q\in \AA(K)$; 
\item 
$ \overline J^{Q}(\mu)=\overline J(\mu)\cdot e^{-2\int_K Qd\mu}$  for {\bf $Q$ {continuous}}. 
\end{enumerate}
\end{lemma}
\begin{proof}
Item 1. follows from 
$$J^Q_{k}(G)\leq \sup_{{\bf a}\in\tilde G_{k}}|VDM_{k}^{Q}({\bf a})|^
{2/k(k-1)}
,$$
and the upper bound (\ref{upboundVDM}) on the limit of the Vandermonde. We prove item 2. and item 3. simultaneously. We first observe that if $\mu \in {\mathcal M}(K)$ and $Q$ is continuous on $K$,  given $\epsilon >0$, there exists a neighborhood $G \subset  {\mathcal M}(K)$ of $\mu$ with
$$\big|\int_{K} Q\Big(d\mu -\frac{1}{k}\sum_{j=1}^{k} 
\delta_{a_{j}}\Big)\big|\leq\epsilon \quad \hbox{for} \ {\bf a} \in \tilde G_k$$
for $k$ sufficiently large. Thus we have
\begin{equation}\label{ineg}
-\epsilon -\int_{K}Qd\mu\leq -\frac{1}{k}\sum_{j=1}^{k}
 Q(a_{j})
 \leq \epsilon-\int_{K}Qd\mu.
 \end{equation}
Note that for $Q\in \AA(K)$, hence lower semicontinuous, we only have the second inequality. 
Since
$$|VDM_k^Q({\bf a})|=|VDM_k({\bf a})|\cdot 
\prod_{j=1}^{k} e^{-(k-1)Q(a_{j})},$$ 
we deduce from (\ref{ineg}) that
$$|VDM_k({\bf a})|e^{-k(k-1)
(\epsilon+\int_K Qd\mu)} 
\leq |VDM^Q_k({\bf a})| \leq 
|VDM_k({\bf a})|e^{k(k-1)(\epsilon-\int_K Qd\mu)}.
$$
Now we take the square, integrate over ${\bf a}\in \tilde G_k$ and take a 
$k(k-1)$-th root of each side to get
$$J_k(G)e^{-2(\epsilon+\int_K Q d\mu)} \leq J^Q_k(G) \leq J_k(G)
e^{2(\epsilon-\int_K Q d\mu)}.
$$
Precisely, given $\epsilon >0$, these inequalities are valid for $G$ a sufficiently small neighborhood of $\mu$. Hence we get, upon taking $\limsup_{k\to \infty}$, the infimum over $G\ni \mu$, and noting that $\epsilon >0$ is arbitrary,
$$\overline J(\mu)=\overline J^Q(\mu)\cdot e^{2\int_K Qd\mu}$$
as desired. If $Q$ is only lower semicontinuous, we still have the upper bounds in the above, which gives item 2.
\end{proof}
From item 1. in Lemma \ref{lem-J-JQ}, we know that for $Q\in \AA(K)$
\begin{equation}\label{Jupbound}
\log \underline J^{Q}(\mu)\leq \log \overline J^{Q}(\mu)\leq -I^Q(\mu).
\end{equation}
In the remainder of this section, we show that
when the measure $\nu$ satisfies a {Bernstein-Markov property}, equalities hold in (\ref{Jupbound}). 

We first consider the unweighted functionals $\underline J$ and $\overline J$ and the case of an equilibrium measure $\mu=\mu_{K,v}$ where $v\in \AA(K)$
\begin{lemma}\label{lem-eq-case}
Let $K$ be non log-polar, $v\in \AA(K)$, and let $\nu\in {\mathcal M}(K)$ such that $(K,\nu,v)$ satisfy a weighted Bernstein-Markov property. Then, 
 \begin{equation}\label{jversion}
 \log \overline J(\mu_{K,v})= \log\underline  J(\mu_{K,v})=-I(\mu_{K,v}).
 \end{equation}
 \end{lemma}
 \begin{proof}
 To prove (\ref{jversion}), we first verify the following.
 \medskip
 
\noindent {\sl Claim: Fix a neighborhood $G$ of $\mu_{K,v}$. For $\eta >0$, define
\begin{equation*}
 A_{k,\eta}:=\{{\bf Z_k}\in K^{k}: |VDM_k^v({\bf Z_k})|^2 \geq 
 ( \delta^v(K) -\eta)^{k(k-1)}\}.
 \end{equation*}
Given a sequence $\{\eta_j\}$ with $\eta_j\downarrow 0$, there exists a $j_{0}$ and a $k_{0}$ such that
\begin{equation}\label{setinclu}  \forall j\geq j_{0},\quad\forall k\geq k_{0},\quad A_{k,\eta_j} \subset \tilde G_{{k}}.
\end{equation}
}
 \medskip
 \noindent 
We prove (\ref{setinclu}) by contradiction: if false, there are sequences $\{k_l\}$ and $\{j_{l}\}$ tending to infinity such that for all $l$ sufficiently large we can find a point ${\bf Z_{k_l}}=(z_1,...,z_{k_l})$ with ${\bf Z_{k_l}}\in 
 A_{k_l,\eta_{j_{l}}} \setminus \tilde G_{k_l}$. But 
 $$\mu^l:=\frac{1}{k_{l}}\sum_{i=1}^{k_{l}}\delta_{z_{i}}\not\in G$$ 
 for $l$ sufficiently large is a contradiction with item 3. of Theorem \ref{sec3} since ${\bf Z_{k_l}}\in A_{k_l,\eta_{j_{l}}}$ and $\eta_{j_l}\to 0$ imply $\mu^l\to \mu_{K,v}$ weakly. This proves the claim.
 \medskip

Fix a neighborhood $G$ of $\mu_{K,v}$ and a sequence $\{\eta_j\}$ with $\eta_j\downarrow 0$. For $j\geq j_{0}$, choose $k=k_j$ large enough so that the inclusion in (\ref{setinclu}) holds true as well as
\begin{equation}\label{probk2}Prob_{k_j}(K^{k_j}\setminus A_{k_j,\eta_j})\leq 
\Big(1-\frac{\eta_j}{2\delta^v(K)}\Big)^{{k_j}({k_j}-1)},
\end{equation}
and
\begin{equation}\label{seq}\Big(1-\frac{\eta_j}{2\delta^v(K)}\Big)^{{k_j}({k_j}-1)}
\to 0 \quad \hbox{as} \quad j\to \infty,
\end{equation}
which is possible (for (\ref{probk2}) we make use of Corollary \ref{johansson}). In view of (\ref{setinclu}), the definition of $Prob_{k_{j}}$, and (\ref{probk2}), we have
\begin{align}\notag
\frac{1}{Z^v_{k_j}} 
\int_{\tilde G_{k_j}}  |VDM_{k_j}^{v}({\bf Z_{k_j}})|^2  d\nu({\bf Z_{k_j}})
& \geq \frac{1}{Z^v_{k_j}} 
\int_{A_{k_j,\eta_j} }  |VDM_{k_j}^{v}({\bf Z_{k_j}})|^2  d\nu({\bf Z_{k_j}}) \\
\label{ineg-for-J}
& \geq 1- \Big(1-\frac{\eta_j}{2\delta^v(K)}\Big)^{{k_j}({k_j}-1)}.
\end{align}
Note that, because of (\ref{seq}), the lower bound in (\ref{ineg-for-J}) tends to 1 as $j\to\infty$. Then,
since $(K,\nu,v)$ satisfy a weighted Bernstein-Markov property, we derive, with 
the asymptotics of $Z_{k_{j}}^{v}$ given in Proposition \ref{weightedtd}, that
$$\liminf_{j\to \infty} \frac{1}{k_j(k_j-1)}\log \int_{\tilde G_{k_j}}  |VDM_{k_j}^{v}({\bf Z_{k_j}})|^2  d\nu({\bf Z_{k_j}}) \geq \log \delta^v(K).$$
Given any sequence of positive integers $\{k\}$ we can find a subsequence $\{k_j\}$ as above corresponding to some $\eta_j\downarrow 0$; hence
$$\liminf_{k\to \infty} \frac{1}{{k}({k}-1)}\log \int_{\tilde G_{k}}  |VDM_{k}^{v}({\bf Z}_k)|^2 d\nu({\bf Z}_{k})\geq \log \delta^v(K).$$
It follows that 
$$
\log \underline J^{v}(G)\geq \log \delta^v(K).$$
Taking the infimum over all neighborhoods $G$ of $\mu_{K,v}$ we obtain
$$\log \underline J^{v}(\mu_{K,v})\geq  \log \delta^v(K).$$
Using item 2. of Lemma \ref{lem-J-JQ} with $\mu = \mu_{K,v}$, we get
$$
\log \underline J(\mu_{K,v})\geq -I(\mu_{k,v}),
$$
and with the unweighted version of item 1., we obtain (\ref{jversion}).
 \end{proof}
 
 \begin{remark} We observe that the proof only used the property 
$$\lim_{k\to \infty} (Z_k^v(K,\nu))^{1/k(k-1)}=\delta^{v}(K).$$
 
 \end{remark}
 \begin{theorem} \label{rel-J-E} 
 Let $K$ be a non log-polar compact subset of $\R^{n}$ 
 and let $\nu\in {\mathcal M}(K)$ satisfy the (strong) Bernstein-Markov property. \\
 (i) For any $\mu\in \mathcal M(K)$, 
 \begin{equation}\label{minunwtd}
 \log \overline J(\mu)= \log\underline  J(\mu)=-I(\mu).\end{equation}
(ii) Let $Q\in\AA(K)$. Then
\begin{equation}\label{rel-J-JQ}
 \overline J^{Q}(\mu)=\overline J(\mu)\cdot e^{-2\int_K Qd\mu},
 \end{equation}
(and with the $\underline J,\underline J^Q$ 
functionals as well) so that,
 \begin{equation}\label{minwtd}\log \overline J^Q(\mu)= \log \underline J^Q(\mu)=-I^{Q}(\mu).
\end{equation}
\end{theorem}
\begin{proof} 
We first prove (i). The upper bound 
\begin{equation}\label{upper-J}
\log\overline J(\mu)\leq -I(\mu)
\end{equation}
 is the unweighted version of (\ref{Jupbound}).
For the lower bound $-I(\mu)\leq\log\underline J(\mu)$ we first assume that 
$I(\mu)<\infty$.  
Using Corollary \ref{corapprox}, 
there exists a sequence of (continuous) functions $Q_{m}$ defined on $K$ and measures $\mu_{m}=\mu_{K,Q_{m}}$ tending weakly to $\mu$ such that, 
\begin{equation}\label{from2}\lim_{m\to \infty}I(\mu_{m}) =I(\mu).
\end{equation}
Thus we can apply Lemma \ref{lem-eq-case} to conclude
$$\log \overline J(\mu_{m})= \log  \underline J(\mu_{m})=-I(\mu_{m}),$$
and from (\ref{from2}) along with the uppersemicontinuity of the functional $\mu \to \underline J(\mu)$, we derive
$$\lim_{m\to \infty} \log \underline J(\mu_{m})=-I(\mu)\leq \log \underline J(\mu).$$
Together with (\ref{upper-J})
we get
$$\log \underline J(\mu)=\log  \overline J(\mu)=-I(\mu).$$
If $\mu\in \mathcal M(K)$ satisfies $I(\mu)=\infty$, 
Item 1. of Lemma \ref{lem-J-JQ} shows that $\overline J(\mu)=0$. 

We next proceed with assertion (ii). For (\ref{rel-J-JQ}), it is sufficient to prove the inequality 
\begin{equation}\label{minor-JQ}
\overline J^{Q}(\mu)\geq\overline J(\mu)\cdot e^{-2\int_K Qd\mu}.
\end{equation}
We first assume that $Q$ is finite $\mu$-almost everywhere and $I(\mu)<\infty$ so that Lemma \ref{lemma-scal} can be applied on $K$. Let $K_{m}$ be the sequence of compact subsets of $K$ and $\mu_{m}$ be the sequence of measures in $\MM(K_{m})$ given by that lemma.
By the upper semicontinuity of the functional $\overline J^{Q}$,
$$\overline J^{Q}(\mu)\geq\limsup_{m \to \infty}\overline J^{Q}(\mu_{m}).$$
Also, by item 3. of Lemma \ref{lem-J-JQ} and (\ref{minunwtd}), since $Q|_{K_m}$ is continuous,
$$\overline J^{Q}(\mu_{m})=\overline J(\mu_{m})e^{-2\int Qd\mu_{m}}=
e^{-I(\mu_{m})-2\int Qd\mu_{m}}.$$
Hence, (\ref{minor-JQ}) follows from items 2. and 3. of Lemma \ref{lemma-scal}.
When $I(\mu)=\infty$, both sides of (\ref{minor-JQ}) equal 0, since $\overline J(\mu)=e^{-I(\mu)}$ and, by definition,
$0\leq\overline J^{Q}(\mu)$. If $\mu(\{Q=\infty\})>0$, this is true as well because 
$\overline J(\mu)>-\infty$ while the exponential in the right-hand side vanishes.

Finally, (\ref{minwtd}) follows from (\ref{minunwtd}) and (\ref{rel-J-JQ}).
 \end{proof}

\begin{remark} \label{fourfour} We note that the Bernstein-Markov property of the measure $\nu$ has only been applied with the sequence of continuous weights $Q_{m}$ that appear when approaching $\mu$ with Corollary \ref{corapprox}. \end{remark}
From now on, we simply use the notation $J,J^Q$, without the overline or underline. It follows from (\ref{minunwtd}) and (\ref{minwtd}) that these functionals are independent of the measure $\nu$; i.e., we have shown: {\it  if $\nu\in \mathcal M(K)$ is {\bf any} (strong) Bernstein-Markov measure, for any $Q\in \AA(K)$,
and for any $\mu\in \mathcal M(K)$ we have}
\begin{equation}\label{w=j=i} 
\log J^Q(\mu)=-I^Q(\mu). 
\end{equation}
\section{Large Deviation Principle in $\R^n$}
Fix a non log-polar compact set $K$ in $\R^{n}$, a measure $\nu$ on $K$
and $Q\in\AA(K)$. 
Define 
$j_k:  K^{k} \to \mathcal M(K)$ via 
\begin{equation}\label{jk} j_k(x_1,...,x_{k})
=\frac{1}{k}\sum_{j=1}^{k} \delta_{x_j}.\end{equation}
The push-forward
$\sigma_k:=(j_k)_*(Prob_k) $ (see (\ref{probk}) for the definition of $Prob_k$) is a probability measure on $\mathcal M(K)$: for a Borel set $G\subset \mathcal M(K)$,
\begin{equation}\label{sigmak}
\sigma_k(G)=\frac{1}{Z_k^{Q}} \int_{\tilde G_{k}} |VDM_k^Q(x_1,...,x_{k})|^2 d\nu(x_1) \cdots d\nu(x_{k}),
\end{equation}
recall (\ref{L2}), (\ref{probk}) and (\ref{nbhddef}); here, $Z_k^{Q}$ depends on $K$, $Q$ and $\nu$.

\begin{theorem} \label{ldp} Assume $\nu$ is a (strong) Bernstein-Markov measure on $K$, $Q\in \mathcal A(K)$, and $\nu$ satisfies 
\begin{equation}\label{nuas}\lim_{k\to \infty} (Z_k^Q(K,\nu))^{1/k(k-1)}=\exp{(-V_w)}=\delta^{Q}(K).\end{equation} The sequence $\{\sigma_k=(j_k)_*(Prob_k)\}$ of probability measures on $\mathcal M(K)$ satisfies a {\bf large deviation principle} with speed $k^{2}$ and good rate function $\mathcal I:=\mathcal I_{K,Q}$ where, for $\mu \in \mathcal M(K)$,
\begin{equation*}
\mathcal I(\mu):=\log J^Q(\mu_{K,Q})-\log J^Q(\mu)=I^Q(\mu)-I^Q(\mu_{K,Q}).
\end{equation*}
 \end{theorem}

This means that $\mathcal I:\mathcal M(K)\to [0,\infty]$ is a lower semicontinuous mapping such that the sublevel sets $\{\mu \in \mathcal M(K): \mathcal I(\mu)\leq \alpha\}$ are compact in the weak topology on $\mathcal M(K)$ for all $\alpha \geq 0$ ($\mathcal I$ is ``good'') satisfying (\ref{lowb}) and (\ref{highb}):

\begin{definition} \label{equivform}
The sequence $\{\mu_k\}$ of probability measures on $\mathcal M(K)$ satisfies a {\bf large deviation principle} (LDP) with good rate function $\mathcal I$ and speed 
$k^2$ if for all 
measurable sets $\Gamma\subset \mathcal M(K)$, 
\begin{equation}\label{lowb}-\inf_{\mu \in \Gamma^0}\mathcal I(\mu)\leq \liminf_{k\to \infty} \frac{1}{k^2} \log \mu_k(\Gamma) \ \hbox{and}\end{equation}
\begin{equation}\label{highb} \limsup_{k\to \infty} \frac{1}{k^2} \log \mu_k(\Gamma)\leq -\inf_{\mu \in \bar \Gamma}\mathcal I(\mu).\end{equation}
\end{definition}

In the setting of $\mathcal M(K)$, to prove a LDP it suffices to work with a base for the weak topology. The following is a special case of a basic general existence result for a LDP given in Theorem 4.1.11 in \cite{DZ}.

\begin{proposition} \label{dzprop1} Let $\{\sigma_{\epsilon}\}$ be a family of probability measures on $\mathcal M(K)$. Let $\mathcal B$ be a base for the topology of $\mathcal M(K)$. For $\mu\in \mathcal M(K)$ let
$$\mathcal I(\mu):=-\inf_{\{G \in \mathcal B: \mu \in G\}}\bigl(\liminf_{\epsilon \to 0} \epsilon \log \sigma_{\epsilon}(G)\bigr).$$
Suppose for all $\mu\in \mathcal M(K)$,
$$\mathcal I(\mu)=-\inf_{\{G \in \mathcal B: \mu \in G\}}\bigl(\limsup_{\epsilon \to 0} \epsilon \log \sigma_{\epsilon}(G)\bigr).$$
Then $\{\sigma_{\epsilon}\}$ satisfies a LDP with rate function $\mathcal I(\mu)$ and speed $1/\epsilon$. 
\end{proposition}

We give our proof of Theorem \ref{ldp} using Theorem \ref{rel-J-E}. 

\begin{proof} As a base $\mathcal B$ for the topology of $\mathcal M(K)$, we 
simply take all open sets. For $\{\sigma_{\epsilon}\}$, we take the sequence of probability measures $\{\sigma_k\}$ on $\mathcal M(K)$ and we take $\epsilon =k^{-2}$. For $G\in \mathcal B$, 
$$\frac{1}{k^{2}}\log \sigma_k(G)= \frac{k-1}{k}\log J_k^Q(G)-\frac{1}{k^{2}}\log Z_k^{Q}$$
using (\ref{jkqmu}) and (\ref{sigmak}). From (\ref{w=j=i}), and the fact that (\ref{nuas}) holds,
$$\lim_{k\to \infty} \frac{1}{k^{2}}\log Z_k^{Q}=\log \delta^Q(K)= \log J^Q(\mu_{K,Q});$$ and by Theorem \ref{rel-J-E}, 
$$\inf_{G \ni \mu} \limsup_{k\to \infty} \log J_k^Q(G)=\inf_{G \ni \mu} \liminf_{k\to \infty} \log J_k^Q(G)=\log J^Q(\mu).$$
Thus by Proposition \ref{dzprop1} and Theorem \ref{rel-J-E}, $\{\sigma_k\}$ satisfies an LDP with rate function 
$$\mathcal I(\mu):=\log J^Q(\mu_{K,Q})-\log J^Q(\mu)=I^Q(\mu)-I^Q(\mu_{K,Q})$$
and speed $k^{2}$. This rate function is good since $\mathcal M(K)$ is compact.
\end{proof}

 \begin{remark} Note that the rate function is independent of the (strong) Bernstein-Markov measure $\nu$ satisfying (\ref{nuas}).

\end{remark}

\section{Measures $\nu$ of infinite mass on $K\subset {\bf S}\subset \R^3$}
In this section, we restrict to the setting of compact subsets of the two-dimensional sphere ${\bf S}$ in $\R^{3}$, of center $(0,0,1/2)$ and radius $1/2$. Then in the following sections, we use stereographic projection from ${\bf S}$ to the complex plane to derive a large deviation principle on unbounded subsets of $\C$. Now typically, on the complex plane, one would like to consider locally finite measures with infinite mass like, e.g., the Lebesgue measure.
We use the stereographic projection $T$ defined in (\ref{stereo}) which sends the north pole $P_{0}=(0,0,1)$ of ${\bf S}$ to the point at infinity in $\C$.
On the sphere ${\bf S}$ we are thus led to consider positive measures $\nu$, locally finite in ${\bf S}\setminus P_{0}$, such that 
\begin{equation}\label{cond-nu-inf}
\nu(V_{P_{0}})=\infty,\quad\text{for all neighborhoods }V_{P_{0}}\text{ of }P_{0}.
\end{equation}
The goal of this section is to extend the results from the previous sections to such measures.

Fix a compact subset $K$ of ${\bf S}$ containing $P_{0}$. To ensure the finiteness of the different quantities defined previously, some condition should be satisfied linking the measure $\nu$ and the increase of the weights $Q$ near $P_{0}$. We assume that
\begin{equation}\label{cond-nu-S}
\exists a>0,\quad\int_{K}\epsilon(x)^{a}d\nu(x)<\infty,
\end{equation}
where $\epsilon(x)$ is some nonnegative continuous function that tends to 0 as $x$ tends to $P_{0}$, and that
\begin{equation}\label{cond-eps-S}
Q(x)\geq-\log\epsilon(x),\quad\text{as }x\to P_{0}.
\end{equation}
This implies in particular that $Q(P_{0})=\infty$. We next state a weighted Bernstein-Walsh lemma on the sphere.
\begin{theorem}\label{BW-S}
Let $K$ be a closed non log-polar subset of ${\bf S}$ and $Q\in \AA(K)$. Let
\begin{equation}\label{pol-S}
p_{k}(x)=\prod_{j=1}^{k}|x-x_{j}|,\qquad x\in {\bf S},
\end{equation}
where $x_{1},\ldots,x_{k}\in {\bf S}$ and assume
$$|p_{k}(x)e^{-kQ(x)}|\leq M\quad\text{for }x\in S_{w}\setminus P\text{ where }P
\text{ is log-polar (possibly empty)}.$$
Then 
$$|p_{k}(x)|\leq M\exp(k(-U^{\mu_{K,Q}}(x)+F_{w})),\qquad x\in {\bf S},$$
and
$$|p_{k}(z)e^{-kQ(z)}|\leq M \quad\text{for }x\in K\setminus \tilde P\text{ where }
\tilde P\text{ is log-polar (possibly empty)}.$$
\end{theorem}
\begin{proof}
Using the stereographic map $T$ defined in (\ref{stereo}), this theorem is a translation of Theorem \ref{BW-C} on $\C$.
\end{proof}
We will also need a lemma related to where the $L^{p}$ norm of a weighted ``polynomial'' lives, see \cite[III, Theorem 6.1]{ST}, \cite[Theorem 6.1]{B2} for polynomials on $\C$. 
For $w=e^{-Q}$, we set
$$S_{w}^{*}=\{x\in K,~U^{\mu_{K,Q}}(x)+Q(x)\leq F_{w}\}.$$
Note that, as $U^{\mu_{K,Q}}(x)+Q(x)$ is lower semicontinuous, $S_{w}^{*}$ is a closed subset of ${\bf S}$ which, moreover, does not contain $P_{0}$. Indeed, $U^{\mu_{K,Q}}(x)$ is bounded below while $Q(x)$ tends to infinity as $x$ tends to $P_{0}$. Moreover, from Theorem \ref{Frost-C}, $S_w\subset S_w^*$.
\begin{lemma}\label{Tom}
Let $p>0$, $K$ a non log-polar compact subset of ${\bf S}$ containing $P_{0}$, $Q\in \AA(K)$  and $\nu$ a positive measure on $K$ satisfying 
(\ref{cond-nu-inf})--(\ref{cond-eps-S}). We assume that $(K,\nu,Q)$ satisfies the weighted Bernstein-Markov property. Let $N\subset K$ be a closed neighborhood of $S_{w}^{*}$. Then, there exists a constant $c>0$ independent of $k$ and $p$ such that, for all expressions $p_{k}$ of the form (\ref{pol-S}),
$$\int_{K}|p_{k}e^{-kQ}|^{p}d\nu\leq(1+\OO(e^{-ck}))\int_{N}|p_{k}e^{-kQ}|^{p}d\nu. $$
\end{lemma}
\begin{proof}
We normalize $p_{k}$ so that $\|p_{k}e^{-kQ}\|_{S_{w}^{*}}=1$. It is sufficient to show that there exists a constant $c>0$ such that for $k$ large,
\begin{equation}\label{ineg-K-N}
\int_{K\setminus N}|p_{k}e^{-kQ}|^{p}d\nu\leq e^{-ck},
\end{equation}
and that, for every $\epsilon>0$ and $k$ large,
$$\int_{K}|p_{k}e^{-kQ}|^{p}d\nu\geq e^{-\epsilon k}.$$

For the second inequality, we use the $L^{p/2}$--Bernstein-Markov property (recall Remark \ref{lp}) which gives
$$\int_{K}|p_{k}e^{-kQ}|^{p}d\nu\geq M_{2k}^{-p/2}\|p_{k}e^{-kQ}\|_{K}^{p}\geq
M_{2k}^{-p/2}\|p_{k}e^{-kQ}\|_{S_{w}^{*}}^{p}=M_{2k}^{-p/2}\geq e^{-\epsilon k},$$
for $k$ large, where we notice that $p_{k}^{2}$ is a real polynomial of degree $2k$. 

For the first inequality, we use Theorem \ref{BW-S}
and the fact that $S_{w}\subset S_{w}^{*}$. This implies that, for $x\in K$,
$$|e^{-kQ(z)}p_{k}(x)|\leq\|e^{-kQ}p_{k}\|_{S_{w}}e^{-k(U^{\mu_{K,Q}}+Q-F_{w})}\leq
e^{-k(U^{\mu_{K,Q}}+Q-F_{w})}.$$
Since $U^{\mu_{K,Q}}$ is bounded below on $K$, there exists a constant $b_{0}$ such that
$$-U^{\mu_{K,Q}}(x)-Q(x)+F_{w}\leq\log\epsilon(x)+b_{0},\quad x\in K,$$
and, as $\epsilon(x)$ tends to 0 as $x$ tends to $P_{0}$, there exists
a neighborhood $V_{P_{0}}$ of $P_{0}$ such that $e^{b_{0}}\epsilon(x)^{1/2}<1$ for $x\in V_{P_{0}}$. On the other hand, since $N$ is a closed neighborhood of $S_{w}^{*}$ and $-U^{\mu_{K,Q}}-Q$ is upper semicontinuous, there exists a constant $b_{1}>0$ such that
$$-U^{\mu_{K,Q}}(x)-Q(x)+F_{w}\leq -b_{1}<0,\quad x\in K\setminus N.
$$
From this we deduce that
\begin{align*}\int_{K\setminus N}|p_{k}e^{-kQ}|^{p}d\nu & =
\int_{V_{P_{0}}}|p_{k}e^{-kQ}|^{p}d\nu+
\int_{K\setminus (N\cup V_{P_{0}})}|p_{k}e^{-kQ}|^{p}d\nu\\
& \leq \int_{V_{P_{0}}}e^{kpb_{0}}\epsilon(x)^{kp}d\nu+e^{-kpb_{1}}\nu(K\setminus V_{P_{0}})\\
& \leq \|e^{b_{0}}\epsilon(x)^{1/2}\|^{kp}_{V_{P_{0}}}\int_{K}\epsilon(x)^{a}d\nu
+e^{-kpb_{1}}\nu(K\setminus V_{P_{0}}),
\end{align*}
for $k$ large, which implies (\ref{ineg-K-N}).
\end{proof}
We are now in a position to prove an extended version of Proposition \ref{weightedtd}.
\begin{proposition}\label{ZQ-inf}
Let $K$ be a non log-polar compact subset of ${\bf S}$ containing $P_{0}$, $Q\in \AA(K)$ and $\nu$ a positive measure on $K$ satisfying 
(\ref{cond-nu-inf})--(\ref{cond-eps-S}). We assume that $(K,\nu,Q)$ satisfies a weighted Bernstein-Markov property. Then the $L^{2}$ normalization constants $Z_{k}^{Q}$ defined in (\ref{L2}) are finite and
$$\lim_{k\to \infty}(Z_{k}^{Q})^{1/k(k-1)}=\delta^{Q}(K).$$
\end{proposition}
\begin{proof}
In view of (\ref{cond-nu-S}) and (\ref{cond-eps-S}), it is clear that, for $k$ large, the integral defining $Z_{k}^{Q}$ is finite.
The lower bound, 
$$ \delta^{Q}(K)\leq \liminf_{k\to \infty}(Z_{k}^{Q})^{1/k(k-1)},$$
is proved as in the proof of Proposition \ref{weightedtd} by making use of the weighted Bernstein-Markov property. For the upper bound, 
\begin{equation}\label{up-Z-S}
\limsup_{k\to \infty}(Z_{k}^{Q})^{1/k(k-1)}\leq\delta^{Q}(K),
\end{equation}
we first note that the expression $|VDM_{k}^{Q}({\bf X_{k}})|^{2}$ is, in each variable, of the form $e^{-2(k-1)Q}|q|^{2}$ with $|q|$ as in (\ref{pol-S}) for $k-1$. Hence, by using Lemma \ref{Tom} with $p=2$ for each of the $k$ variables, and with $N\subset K$ a closed neighborhood of $S_{w}^{*}$ as in Lemma \ref{Tom} with $\nu(N)<\infty$ and $P_0\not \in N$, we get 
\begin{align*}
Z_{k}^{Q}=\int_{K^{k}}|VDM_{k}^{Q}({\bf X_{k}})|^{2}d\nu({\bf X_{k}}) & \leq
(1+\OO(e^{-c(k-1)}))^{k}\int_{N^{k}}|VDM_{k}^{Q}({\bf X_{k}})|^{2}d\nu({\bf X_{k}})\\
& \leq (1+\OO(e^{-c(k-1)}))^{k}(\delta^{Q}_{k}(K))^{k(k-1)}\nu(N)^{k},
\end{align*}
which implies (\ref{up-Z-S}) by taking the $k(k-1)$-th root and letting $k$ go to infinity.
\end{proof}
The next goal is to generalize Corollary \ref{johansson}. \begin{corollary}
We assume that the conditions (\ref{cond-nu-inf})--(\ref{cond-eps-S}) are satisfied and that $(K,\nu,Q)$ satisfies the weighted Bernstein-Markov property. 
Then, with the notation of Corollary \ref{johansson}, there exist a constant $c>0$ and $k^*=k^*(\eta)$ such that for all $k>k^*$, 
\begin{equation}\label{ineg-Prob}
Prob_k(K^{k}\setminus A_{k,\eta})\leq 
\Big(1-\frac{\eta}{2 \delta^{Q}(K)}\Big)^{k(k-1)}\nu(N)^{k}+\OO(e^{-ck}),
\end{equation}
where $N\subset K$ is a closed neighborhood of $S_{w}^{*}$ as in Lemma \ref{Tom} with $\nu(N)<\infty$ and $P_0\not \in N$.
\end{corollary}
\begin{proof}
We set
$B_{k,\eta}:=K^{k}\setminus A_{k,\eta}$ and decompose the integral in
$$Prob_k(K^{k}\setminus A_{k,\eta})=\frac{1}{Z_{k}^{Q}}\int_{B_{k,\eta}}|VDM_{k}^{Q}({\bf X_{k}})|^{2}d\nu({\bf X_{k}})$$
as a sum of two integrals over $B_{k,\eta}\cap N^{k}$ and $B_{k,\eta}\cap (K^{k}\setminus N^{k})$. Recalling the definition of the set $A_{k,\eta}$, the first term is less than
$$\Big(1-\frac{\eta}{2 \delta^{Q}(K)}\Big)^{k(k-1)}\nu(N)^{k},$$
for $k$ large. The second term is less than
$$\frac{1}{Z_{k}^{Q}}\int_{K^{k}\setminus N^{k}}|VDM_{k}^{Q}({\bf X_{k}})|^{2}d\nu({\bf X_{k}})\leq\sum_{j=1}^{k}\frac{1}{Z_{k}^{Q}}\int_{U_{j}}|VDM_{k}^{Q}({\bf X_{k}})|^{2}d\nu({\bf X_{k}}),$$
where $U_{j}=K\times\cdots\times(K\setminus N)\times\cdots K$ and the subset $K\setminus N$ is in $j$-th position.
As already observed in the previous proof, the expression $|VDM_{k}^{Q}({\bf X_{k}})|^{2}$ is, in each variable, of the form $e^{-2(k-1)Q}|q|^{2}$ with $|q|$ as in (\ref{pol-S}) for $k-1$. Hence, applying Lemma \ref{Tom} with the $j$-th variable to the integral over $U_{j}$, we get the upper bound
$$\OO(e^{-c(k-1)})\sum_{j=1}^{k}\frac{1}{Z_{k}^{Q}}\int_{V_{j}}|VDM_{k}^{Q}({\bf X_{k}})|^{2}d\nu({\bf X_{k}}),$$
where $V_{j}=K\times\cdots\times N\times\cdots K$. Replacing $N$ with $K$ we finally get the upper bound $\OO(ke^{-c(k-1)})$,
which implies (\ref{ineg-Prob}) with a different $c$.
\end{proof}
The last result that needs to be extended is the first item of Lemma \ref{lem-J-JQ}, namely, that for $Q\in \AA(K)$, and $\nu$ satisfying (\ref{cond-nu-inf})--(\ref{cond-eps-S}),
\begin{equation}\label{ineq-J-infty}
\overline J^{Q}(\mu)\leq e^{-I^Q(\mu)}.
\end{equation}
With the notation of Section \ref{J-funct}, we remark that fixing $a >0$ as in (\ref{cond-nu-S}), we can write
$$\int_{\tilde G_{k}}|VDM^{Q}_{k}({\bf a})|^{2}d\nu({\bf a})=\int_{\tilde G_{k}}|VDM^{Q_{k}}_{k}({\bf a})|^{2}d\tilde\nu({\bf a}),$$
where 
$$Q_{k}(x)=Q(x)-\frac{a}{2(k-1)}Q_{+}(x),\quad
\quad\tilde\nu(x)=e^{-a Q_{+}(x)}\nu(x),$$
and $Q_{+}=\max(Q,0)$.
Observe that $\{Q_{k}\}_{k}$ is an increasing sequence of admissible weights that converges pointwise to $Q$ as $k$ tends to infinity. Also, in view of (\ref{cond-nu-S}) and (\ref{cond-eps-S}), $\tilde\nu$ is a finite measure. Since
$$\int_{\tilde G_{k}}|VDM^{Q_{k}}_{k}({\bf a})|^{2}d\tilde\nu({\bf a})\leq
\int_{\tilde G_{k}}|VDM^{Q_{k_{0}}}_{k}({\bf a})|^{2}d\tilde\nu({\bf a}),\quad k\geq k_{0},$$
we have
$$\int_{\tilde G_{k}}|VDM^{Q}_{k}({\bf a})|^{2}d\nu({\bf a})\leq \int_{\tilde G_{k}}|VDM^{Q_{k_{0}}}_{k}({\bf a})|^{2}d\tilde\nu({\bf a}),\quad k\geq k_{0}.$$
By letting $k$ go to infinity in this inequality and taking the infima over all neighborhoods $G$ of a measure $\mu$ in $\MM(K)$, we obtain
$$J^{Q}_{\nu}(\mu)\leq J^{Q_{k_{0}}}_{\tilde\nu}(\mu),$$
where, here, the subscript denotes the measure with respect to which the Vandermonde is integrated. Since $\tilde\nu$ is of finite mass, we derive from item 1. of Lemma \ref{lem-J-JQ} that
$$J^{Q}_{\nu}(\mu)\leq e^{-I^{Q_{k_{0}}}(\mu)}.$$
Letting $k_{0}$ go to infinity, and making use of
$$\int Q_{k_{0}}d\mu\to\int Qd\mu,\quad\text{as }k_{0}\to\infty,$$
which follows from the monotone convergence theorem,
we obtain (\ref{ineq-J-infty}).

From the results above and the proofs of the previous sections, one may check that
Theorem \ref{rel-J-E} extends to the measures $\nu$ considered in this section.
For future reference, we state this as a theorem.
\begin{theorem}\label{rel-JE-S}
Let $K$ be a non log-polar compact subset of the sphere ${\bf S}$ in $\R^{3}$, containing $P_{0}$, $Q\in \AA(K)$ and $\nu$ a positive measure on $K$ satisfying (\ref{cond-nu-inf})--(\ref{cond-eps-S}). Assume $\nu$ satisfies a (strong) Bernstein-Markov property. Then,
$$\log\overline J^{Q}(\mu)=\log\underline J^{Q}(\mu)=-I^{Q}(\mu).$$
\end{theorem} 
Also, the large deviation principle asserted in
Theorem \ref{ldp} extends.
\begin{theorem}\label{ldp-S}
Let $K$ be a compact subset of the sphere ${\bf S}$ in $\R^{3}$, containing $P_{0}$, $Q\in \AA(K)$ and $\nu$ a positive measure on $K$ satisfying (\ref{cond-nu-inf})--(\ref{cond-eps-S}) and (\ref{nuas}). Then, the large deviation principle from Theorem \ref{ldp} holds.
\end{theorem} 

\begin{remark} In Theorems \ref{rel-JE-S} and \ref{ldp-S}, the conclusion is valid for any $\nu$ satisfying a (strong) Bernstein-Markov property and (\ref{cond-nu-inf})--(\ref{cond-eps-S}) (with any appropriate function $\epsilon(x)$). Moreover, the rate function in Theorem \ref{ldp-S} is independent of $\nu$.

\end{remark}

\section{The $J^{Q}$ functionals on unbounded sets in $\C$} 
We return to the case of unbounded sets in $\C$; our goal is to use Theorems \ref{rel-JE-S} and \ref{ldp-S} to derive their versions in our current setting. In the sequel we will need the Bernstein-Markov property on $\C$. For $K$ a closed subset of $\C$, $\nu$ a positive measure on $K$, locally finite but possibly of infinite mass in a neighborhood of infinity, and $Q$ a weakly admissible weight as in (\ref{weak-adm}), we say that $(K,\nu,Q)$ satisfies the Bernstein-Markov property if
\begin{equation}\label{unboundedBM}
\forall p_{k}\in\PP_{k}(\C),\quad\|e^{-kQ}p_{k}\|_{K}\leq M_{k}\|e^{-kQ}p_{k}\|_{L^{2}(\nu)},\quad \text{with }\limsup_{k\to \infty}M_{k}^{1/k}=1.
\end{equation}
As in Section \ref{Sec-BM}, if (\ref{unboundedBM}) holds true for any continuous weakly admissible weight $Q$, we will say that $\nu$ satisfies a strong Bernstein-Markov property. Note that the polynomials $p_{k}$ in (\ref{unboundedBM}) are polynomials with respect to the \emph{complex variable}~$z$.
\begin{example}\label{example-BM}
For $Q$ a continuous admissible weight on $\R$, the linear Lebesgue measure $d\lambda$ provides an example of a measure with unbounded support satisfying (\ref{unboundedBM}). Indeed, from Theorem \ref{BW-C}, the sup norm of $e^{-kQ}p_{k}$ is attained on $S_{w}$ which is compact. Hence, it suffices to prove (\ref{unboundedBM}) on $S_{w}$ or any compact set containing $S_{w}$, for instance a finite interval $I$ (see Remark \ref{strongbm}). For $Q$ a continuous admissible weight on $\C$,
similar reasoning shows that planar Lebesgue measure $dm$ on $\C$ satisfies
(\ref{unboundedBM}) as well (in this case one considers the restriction of $dm$ to a closed disk).

Next, let $Q$ be an admissible weight on $K=\{x\in \R: x\geq 0\}$ which is continuous except $Q(0)=+\infty$. In this case, property (c) of Theorem I.1.3 \cite{ST} shows that $S_w$, the support of $\mu_{K,Q}$, will be compact and disjoint from the origin. Thus linear Lebesgue measure similarly satisfies (\ref{unboundedBM}). Specific examples are Laguerre weights $Q(x)=\lambda x -s\log x$ with $\lambda,s>0$ which occur in the Wishart ensemble (see \cite{hpbook}, section 5.5); and $Q(x)=c(\log x)^2$ with $c\geq 0$, occurring in the Stieltjes-Wigert ensemble (see \cite{[T]} and \cite{[ESS]}).

\end{example}
\begin{remark}\label{BM-Simeo}
For future use we observe that, if $\nu$ has infinite mass in a neighborhood of infinity, then the Bernstein-Markov property (\ref{unboundedBM}) is automatically satisfied if we restrict, for each $k$, to polynomials $p_{k}$ of \emph{exact degree} $k$ and the weight $Q$ satisfies a condition slightly stronger than weak admissibility (\ref{weak-adm}), namely
\begin{equation}\label{simeonov}
\lim_{z\in K,~|z|\to\infty}(Q(z)-\log|z|)=M<\infty.
\end{equation}
Indeed, for $p_{k}$ a monic polynomial of degree $k$, $|e^{-kQ(z)}p_{k}(z)|$ behaves like the constant $e^{-kM}>0$ as $z\to\infty$, so that its sup norm on $K$ is finite while its $L^{2}(\nu)$-norm is infinite. Hence we can take $M_{k}=1$ for each $k\geq 0$.
\end{remark}

Next, we define the (weighted) $L^{2}$ normalization constants for a closed subset $K$ of $\C$, $Q$ weakly admissible and $\nu$ a positive measure on $K$,
\begin{equation}\label{L2-C}
Z_{k}^{Q}(K,\nu):=\int_{K^{k}}|VDM_{k}^{Q}({\bf Z_k})|^{2}d\nu({\bf Z_k}),
\end{equation}
where ${\bf Z_k}:=(z_1,...,z_k)\in K^{k}$. Then we have the correspondence
$$Z_{k}^{Q}(K,\nu)=Z_{k}^{\tilde Q}(T(K),T_{*}\nu)$$
where $\tilde Q$ is defined in (\ref{tildeq}). 
To ensure the finiteness of $Z_{k}^{Q}(K,\nu)$ in case $\nu$ has infinite mass in a neighborhood of infinity, like, e.g., Lebesgue measure, we assume that
the weight $Q$ and the measure $\nu$ satisfy conditions that correspond via the inverse of $T$ to the conditions (\ref{cond-nu-S}) and (\ref{cond-eps-S}) on the sphere, namely,
\begin{equation}\label{cond-nu}
\exists a>0,\quad\int_{K}\epsilon(z)^{a}d\nu(z)<\infty,
\end{equation}
and
\begin{equation}\label{cond-eps}
Q(z)-\log|z|\geq-\log\epsilon(z),\quad\text{as }z\to\infty,
\end{equation}
where $\epsilon(z)$ is some nonnegative continuous function that tends to 0 as $z$ tends to $\infty$. Note that the weight $Q$ is then admissible in the sense of \cite{ST} or 2. of Definition \ref{admit}, that is
\begin{equation}\label{admST}
Q(z)-\log|z|\to\infty,\quad\text{as }z\to\infty.
\end{equation}
Using the inequality
$$|z_{i}-z_{j}|\leq(1+|z_{i}|)(1+|z_{j}|),$$
one may also check directly that the $Z_{k}^{Q}(K,\nu)$, $k$ large, are, indeed, finite.

In the typical example where $K=\R$ or $K=\C$ and $\nu$ is Lebesgue measure, $\epsilon(z)$ can be chosen as $|z|^{-\epsilon}$, $\epsilon>0$, and (\ref{cond-eps}) becomes the following \emph{strong admissibility} condition (recall 3. of Definition \ref{admit}):
$$
Q(z)-\log|z|\geq\epsilon\log|z|,\quad\text{as }z\to\infty.
$$

Our next result is a version of Propositions \ref{weightedtd} and \ref{ZQ-inf} on the $k(k-1)$-th root asymptotic behavior of the $L^{2}$ normalization constants for $K$ a closed subset of $\C$.
\begin{proposition}\label{ZQ-C}
Let $K$ be a nonpolar closed subset of $\C$ and $\nu$ a positive measure on $K$. Let $Q$ be a weight on $K$ which is weakly admissible if $\nu$ has finite mass and such that (\ref{cond-nu}) and (\ref{cond-eps}) are satisfied for some function $\epsilon(z)$ if $\nu$ has infinite mass in a neighborhood of infinity. We assume that $(K,\nu,Q)$ satisfies the weighted Bernstein-Markov property (\ref{unboundedBM}).
Then,
$$
\lim_{k\to \infty}(Z_{k}^{Q})^{1/k(k-1)}=\delta^{Q}(K).$$
\end{proposition}
\begin{proof}
Via the inverse map of $T$, the statement is essentially a simple translation of Propositions \ref{weightedtd} and \ref{ZQ-inf}. The only observation to be made is that,
for the proof of Proposition \ref{weightedtd} on the sphere, it suffices that 
a weighted Bernstein-Markov property is satisfied with respect to polynomials $p$ of the particular form
\begin{equation}\label{special-pol}
p(x)=\prod_{j=1}^{k}|x-T(z_{j})|^{2},\quad x\in {\bf S},
\end{equation}
and that this Bernstein-Markov property corresponds to (\ref{unboundedBM}) via $T^{-1}$. 
Also, in the proof of Proposition \ref{ZQ-inf}, it is sufficient to use a version of Lemma \ref{Tom} which only assumes the Bernstein-Markov property for polynomials of the form (\ref{special-pol}) (and thus only holds for such polynomials).
\end{proof}
Weighted $J$-functionals $\underline J^{Q}(\mu)$ and $\overline J^{Q}(\mu)$ can be defined on the closed subset $K$ of $\C$, with respect to a positive measure $\nu$ in $K$, as was done on compact subsets of $\R^{n}$, see Definition \ref{jwmuq}. Then,
$$\underline J^{Q}(\mu)=\underline J^{\tilde Q}(T_{*}\mu),\quad
\overline J^{Q}(\mu)=\overline J^{\tilde Q}(T_{*}\mu),$$
where the $J$-functionals on the right-hand sides involve integrals with respect to the measure $T_{*}\nu$. From this correspondence, and Theorems \ref{rel-J-E} and \ref{rel-JE-S}, we derive the following.
\begin{theorem}\label{85}
With the hypotheses of Proposition \ref{ZQ-C} and assuming that $\nu$ satisfies a strong Bernstein-Markov property on $K$, we have 
$$\log\underline J^{Q}(\mu)=\log\overline J^{Q}(\mu)=-I^{Q}(\mu).$$
\end{theorem}
\begin{proof}
The statement is a translation of Theorems \ref{rel-J-E} and \ref{rel-JE-S} on the sphere. Again, we observe that the strong Bernstein-Markov property for polynomials of the form (\ref{special-pol}) is sufficient for their proofs. We also use 
the equality of the weighted logarithmic energies (\ref{enerK}) and (\ref{enerTK}).
\end{proof}

 \noindent The conclusion is valid for all $\nu$ satisfying the hypotheses; in particular, the functional $J^Q\ (=\underline J^{Q}=\overline J^{Q}) \ =e^{-I^Q}$ for any such $\nu$.
\section{Large deviation principle for unbounded sets in $\C$} 
 A large deviation principle in the spirit of Theorem \ref{ldp} for compact subsets of $\C^m, \ m\geq 1$ has been obtained in \cite{PELD} using the methods of this paper (see also \cite{VELD}). Yattselev \cite{Y} has proved an LDP associated to a specific type of weight on $\C$;  
he uses Lebesgue measure on $\C$. The large deviation principle for strongly admissible weights $Q$ on all of $\C$ with Lebesgue measure can be found in the book of Hiai and Petz \cite{hpbook}. There they extend the method of Ben Arous and Guionnet \cite{BG}.   
Here, we will utilize the results from the previous sections to establish a LDP in the setting of a closed set $K$ in $\C$, not necessarily bounded, with a weakly admissible weight $Q$ and an appropriate Bernstein-Markov measure. 
The proof is based on a standard contraction principle in LDP theory: 
\begin{theorem}[{\cite[Theorem 4.2.1]{DZ}}]\label{contract}
If $\{P_n\}$ is a sequence of probability measures on a Polish space $X$ satisfying an LDP with speed $\{a_n\}$ and rate function ${\mathcal  I}$, $Y$ is another Polish space and $f:X \to Y$ is a continuous map, then $\{Q_n:=f_*P_n\}$ satisfies an LDP on $Y$ with the same speed and with rate function 
\begin{equation}\label{contract} {\mathcal  J}(y):=\inf \{{\mathcal  I}(x):x\in X, 
                   f(x)=y\}. 
\end{equation} 
\end{theorem}                   
For $K$ a closed, possibly unbounded, subset of $\C$, $\nu$ a locally finite measure on $K$, and $Q$ a weakly admissible weight on $K$, we define the measure $Prob_{k}$ on $K^{k}$ as in (\ref{probk}) for $K$ in $\R^{n}$ {and $j_k:K^k\to \mathcal M(K)$ as in (\ref{jk}).

The statement of the large deviation principle is as follows.
\begin{theorem} \label{ldp2} 
Assume $(K,\nu,Q)$ satisfies the weighted Bernstein-Markov property (\ref{unboundedBM}).
If $\nu$ has finite mass, we also assume that $(K,\nu)$ satisfies a strong Bernstein-Markov property while if $\nu$ has infinite mass in a neighborhood of infinity, we assume that (\ref{cond-nu}) and (\ref{cond-eps}) are satisfied for some function $\epsilon(z)$.
Then the sequence $\{\sigma_k=(j_k)_*(Prob_k)\}$ of probability measures on $\mathcal M(K)$ satisfies a {\bf large deviation principle} with speed $k^{2}$ and good rate function $\mathcal I:=\mathcal I_{K,Q}$ where, for $\mu \in \mathcal M(K)$,
\begin{equation}\label{rate}
\mathcal I(\mu):=\log J^Q(\mu_{K,Q})-\log J^Q(\mu)=I^Q(\mu)-I^Q(\mu_{K,Q}).
\end{equation}
 \end{theorem}
 
  \noindent We emphasize again that, as in Theorem \ref{ldp-S}, the rate function is independent of the measure $\nu$. 
 \begin{proof}
We apply Theorem \ref{contract} to the homeomorphism $f=(T^{-1})_*$: thus to prove an LDP in the setting of a closed set $K$ in $\C$, not necessarily bounded, with a weakly admissible weight $Q$, it suffices, via this contraction principle, to use an LDP in the setting of a compact set $T(K)$ in ${\bf S}\subset \R^3$ with the admissible weight $\tilde Q$. This we have from Theorems \ref{ldp} and \ref{ldp-S}. 
%
In case $\nu$ is
of infinite mass in a neighborhood of infinity, we observe that the strong Bernstein-Markov property of $(K,\nu)$ is not needed.
Indeed, the corresponding Bernstein-Markov property on $T(K)$ is only needed for polynomials that are Vandermonde expressions, hence of maximal degree, and for the weights $Q_{m}$ appearing in Corollary \ref{corapprox}. These weights are of the form $-U^{\mu_{m}}$ with $\mu_{m}\in\MM(T(K))$ and the corresponding weights on $K$ satisfy (\ref{simeonov}) with 
$M=-U^{\mu_{m}}(P_{0})<\infty$,
so that Remark \ref{BM-Simeo} applies.
Finally, the rate function $\II(\mu)$ is good because $I^{Q}(\mu)=I^{\tilde Q}(T_{*}\mu)$, the energy $I^{\tilde Q}$ is lower semicontinuous on the compact set $\MM({\bf S})$, and $T_{*}$ is a homeomorphism.
\end{proof}

\begin{remark}\label{appldp} In particular (see Example \ref{example-BM}), we have a large deviation principle on $K=\{x\in \R: x\geq 0\}$ with Lebesgue measure for the Laguerre weights as well as for the weights occurring in the Stieltjes-Wigert ensembles.

\end{remark}
\section{Applications: $\beta$ ensembles}\label{appli}
Let $K$ be a closed subset of $\C$, $\nu$ a positive measure on $K$, and $Q$ a weakly admissible weight on $K$.
Classical models in random matrix theory involve probability distributions on $K^{k}$ of the form
\begin{equation}\label{distr-rmt}
\frac{1}{\hat Z_{\beta,k}^{Q}}\prod_{1\leq i<j\leq k}|z_{i}-z_{j}|^{2\beta}\prod_{i=1}^{k}e^{-2kQ(z_{i})} d\nu (z_i),
\end{equation}
where $\beta>0$ and the normalization constant $\hat Z_{\beta,k}^{Q}$ is 
\begin{equation}\label{free-ener}
\hat Z_{\beta,k}^{Q}=\int_{K^{k}}\prod_{1\leq i<j\leq k}|z_{i}-z_{j}|^{2\beta}
\prod_{i=1}^{k}e^{-2kQ(z_{i})}d\nu(z_{i}).
\end{equation}
(We caution the reader that in \cite{H} and \cite{AGZ} the $2\beta$ is replaced by $\beta$). The probability distribution (\ref{distr-rmt}) and normalization constant $\hat Z_{\beta,k}^{Q}$ differ from the distribution in (\ref{probk}) and the $L^{2}$ normalization constant $Z_{k}^{Q}$, defined in (\ref{L2-C}), by the exponent $\beta$ and an additional factor
$\prod_{i}e^{-2Q(x_{i})}$ in its integrand.
One may check that all results from the previous sections remain true, with appropriate modifications, when we consider (\ref{distr-rmt}) and (\ref{free-ener}).
Actually, writing the products in (\ref{distr-rmt}) and (\ref{free-ener}) as the square of a  weighted Vandermonde to the power $\beta$, the main modification consists in replacing the weight $Q$ with the weight $Q/\beta$ and to
use the Bernstein-Markov property in $L^{\beta}$ instead of $L^{1}$ as was done in Section \ref{Sec-BM}.
To be precise, because of the factor $k$, instead of $k-1$, in the exponential factors of (\ref{distr-rmt}) and (\ref{free-ener}), the Bernstein-Markov property to be satisfied for a given weight $Q$ here is
\begin{equation}\label{BM-free}
\forall p_{k}\in\PP_{k}(\C),\quad\|e^{-(k+1)Q}p_{k}\|_{K}\leq M_{k}\|e^{-(k+1)Q}p_{k}\|_{L^{2}(\nu)},\quad \text{with }\limsup_{k\to \infty}M_{k}^{1/k}=1.
\end{equation}
This property is slightly weaker than (\ref{unboundedBM}) as it concerns only polynomials in $\PP_{k}(\C)$ instead of $\PP_{k+1}(\C)$, but it will make a minor difference in the assumptions of the large deviation principle because Remark \ref{BM-Simeo} no longer applies.

When $\nu$ has infinite mass in a neighborhood of infinity, the conditions (\ref{cond-nu}) and (\ref{cond-eps}) become
\begin{equation}\label{cond-nu-free}
\exists a>0,\quad\int_{K}\epsilon(z)^{a}d\nu(z)<\infty,
\end{equation}
and
\begin{equation}\label{cond-eps-free}
Q(z)-\beta\log|z|\geq-\log\epsilon(z),\quad\text{as }z\to\infty,
\end{equation}
where $\epsilon(z)$ is some nonnegative continuous function that tends to 0 as $z$ tends to $\infty$. 

Based on the above remarks, one may check that we have the following analogue of Proposition  \ref{ZQ-C} concerning the asymptotics of $\hat Z_{\beta,k}^{Q}$.

\begin{proposition}
Let $K$ be a nonpolar closed subset of $\C$ and $\nu$ a positive measure on $K$. Let $Q$ be a weight on $K$ such that $Q/\beta$ is weakly admissible if $\nu$ has finite mass and such that (\ref{cond-nu-free}) and (\ref{cond-eps-free}) are satisfied for some function $\epsilon(z)$ if $\nu$ has infinite mass in a neighborhood of infinity. We assume that $(K,\nu,Q/\beta)$ satisfies the weighted Bernstein-Markov property (\ref{BM-free}).
Then
$$
\lim_{k\to \infty}(\hat Z_{\beta,k}^{Q})^{1/k(k-1)}=(\delta^{Q/\beta}(K))^{\beta}.$$
\end{proposition}
The following large deviation principle, an analogue of Theorem \ref{ldp2}, also holds true.
\begin{theorem} \label{ldp-free} 
Let $Q/\beta$, $\beta>0$, be a weakly admissible weight on $K$ such that $(K,\nu,Q/\beta)$ satisfies the weighted Bernstein-Markov property (\ref{BM-free}).
We assume that $(K,\nu)$ satisfies a strong Bernstein-Markov property and, in addition, if $\nu$ has infinite mass in a neighborhood of infinity, we assume that (\ref{cond-nu-free}) and (\ref{cond-eps-free}) are satisfied for some function $\epsilon(z)$.
Then the sequence 
of probability measures $\tilde\sigma_{k}$ on $\mathcal M(K)$, defined so that for a Borel set $G\subset\MM(K)$,
$$\tilde\sigma_{k}(G):=\frac{1}{\hat Z_{\beta,k}^{Q}}\int_{\tilde G_k}
\prod_{1\leq i<j\leq k}|z_{i}-z_{j}|^{2\beta}
\prod_{i=1}^{k}e^{-2kQ(z_{i})}d\nu(z_{i}),$$ 
satisfies a {\bf large deviation principle} with speed $k^{2}$ and good rate function $\mathcal I_{K,Q}^{\beta}$ defined by 
$$
\II^{\beta}_{K,Q}(\mu):=I_{\beta}^{Q}(\mu)- I_{\beta}^{Q}(\mu_{K,Q/\beta}),
$$
where
$$I_{\beta}^{Q}(\mu)=\int_{K}\int_{K}\log\frac{1}{|z-t|^{\beta}}d\mu(z)d\mu(t)+
2\int_{K}Q(z)d\mu(z)=\beta I^{Q/\beta}(\mu),\qquad\mu\in\MM(K).$$
 \end{theorem}
 \begin{proof}
 One checks that all arguments in the proof of Theorem \ref{ldp2} go through when considering the probability distribution (\ref{distr-rmt}) instead of the one in (\ref{probk}). In particular, this entails verifying the analogue of Theorems \ref{rel-J-E} and \ref{85}, namely that, for appropriate assumptions on $\nu$ and $Q$, one has
 $$\log \overline J_{\beta}^{Q}(\mu)=\log \underline J_{\beta}^{Q}(\mu)=-I_{\beta}^{Q}(\mu),$$
 where the functionals $\overline J_{\beta}^{Q}$ and $\underline J_{\beta}^{Q}$ are derived from
$$J^Q_{\beta,k}(G):=\Big[\int_{\tilde G_{k}}|VDM^{Q/\beta}_k({\bf a})|^{2\beta}d\nu ({\bf a})\Big]^{1/k(k-1)},\qquad G\subset\MM(K),$$
in the same way as in Definition \ref{jwmuq}. Since Remark \ref{BM-Simeo} does not apply for the Bernstein-Markov property (\ref{BM-free}), the strong Bernstein-Markov property is needed even if $\nu$ has infinite mass.
 \end{proof}
As an example, we take $K=\R$, $d\nu=d\lambda=$ Lebesgue measure on $\R$, and $Q$ a continuous weight such that there exists $\beta'>\beta$ with $Q/\beta'$ weakly admissible:
\begin{equation}\label{q-beta}
\exists M>-\infty,\qquad\liminf_{|z|\to\infty, \ z\in K}(Q(z)-\beta'\log|z|)=M.
\end{equation}
Note that this implies that $Q/\beta$ is admissible. Also, (\ref{cond-nu-free}) and (\ref{cond-eps-free}) hold true with $\epsilon(z)=|z|^{\beta-\beta'}$. The triple 
$(\R,d\lambda,Q/\beta)$ satisfies the weighted Bernstein-Markov property since $Q/\beta$ is admissible, cf., Example \ref{example-BM}. The measure $d\lambda$ likely also satisfies a strong Bernstein-Markov property, but, as already mentioned in the proof of Theorem \ref{ldp2}, for an LDP it is sufficient that this property is satisfied for weights which correspond via $T$ to continuous weights on the sphere ${\bf S}$ of the form $-U^{\mu_{m}}$ where $\mu_{m}\in\MM(T(\R))$ are the measures from Corollary \ref{corapprox}. Moreover, the proof of Lemma \ref{lemma-scal} shows that the supports of the measures $\mu_{m}$ can be chosen to avoid a neighborhood (depending on $m$) of the north pole $P_{0}$ so that the push-backward measures $\nu_{m}=T_{*}^{-1}\mu_{m}$ have compact supports in $\C$. Using the relations (\ref{pot-C-S}) and (\ref{tildeq}), what is then needed is that $d\lambda$ satisfies the Bernstein-Markov property for continuous weights of the form
$$Q_{m}(z)=-U^{\nu_{m}}(z)-\frac12\int\log(1+|t|^{2})d\nu_{m}(t),$$
where $\nu_{m}\in\MM(\R)$ has compact support. These weights are weakly admissible. Hence, by Theorem \ref{BW-C} and the continuity of $Q_{m}$, the corresponding weighted polynomials attain their sup norm on $S_{w_{m}}$ (where $w_{m}=e^{-Q_{m}}$), which is equal to the support of $\nu_{m}$, see Lemma \ref{lem-non-adm} (or more precisely its analogue in $\C$). Consequently, we need that $d\lambda$ satisfies the Bernstein-Markov property for continuous weights on a compact set, which we know holds true (cf., the discussion in Example \ref{example-BM}). Thus we conclude that
the large deviation principle asserted in Theorem \ref{ldp-free} applies on the real line for $d\nu=d\lambda$ and any continuous weight $Q$ satisfying (\ref{q-beta}). We note that this includes the large deviation principle for the law of the spectral measure of Gaussian Wigner matrices (\cite{BG}, \cite[Theorem 2.6.1]{AGZ}) as well as the refined version for weakly confining potentials given in \cite{H}.

When $K=\C$, $d\nu=dm=$ planar Lebesgue measure, and $Q$ is a continuous weight on $\C$ satisfying the growth condition (\ref{q-beta}), assumptions (\ref{cond-nu-free}) and (\ref{cond-eps-free}) still hold true with $\epsilon(z)=|z|^{\beta-\beta'}$. Moreover, the measure $dm$ satisfies the required Bernstein-Markov properties. Hence, Theorem \ref{ldp-free} applies when $K=\C$, $d\nu=dm$, and $Q$ is a continuous weight satisfying (\ref{q-beta}).


\end{document}